\documentclass[12pt]{amsart}
\usepackage{amssymb}
\usepackage{amscd}
\usepackage{color}
\usepackage[USenglish]{babel}
\usepackage[nodayofweek,level]{datetime}

\newcommand{\Diff}{\mathcal{D}}
\newcommand{\Lie}{\mathcal{L}}
\newcommand{\Null}{\mathcal{N}}
\newcommand{\Ann}{\mathcal{A}}

\newcommand{\id}{\text{id}}

\newcommand{\dist}{\delta}
\newcommand{\smally}{\epsilon}
\newcommand{\grad}{\nabla}
\newcommand{\Laplacian}{\Delta}
\DeclareMathOperator{\ad}{ad} \DeclareMathOperator{\Ad}{Ad}

\DeclareMathOperator{\diver}{div} 

\newcommand{\contact}{\theta}
\newcommand{\Reeb}{E}
\newcommand{\onedparam}{\alpha}
\newcommand{\Diffs}{\Diff^s}
\newcommand{\Diffsemi}{\widetilde{\Diff}}
\newcommand{\Diffssemi}{\widetilde{\Diffs}}
\newcommand{\contactlap}{\widetilde{\Laplacian_{\contact}}}
\newcommand{\Diffcontact}{\Diff_{\contact}}
\newcommand{\Diffscontact}{\Diff^s_{\contact}}
\newcommand{\Diffexcontact}{\Diff_{q}}
\newcommand{\Diffsexcontact}{\Diff^s_q}

\newcommand{\Diffham}{\Diff_{\text{Ham}}}

\newcommand{\Difftilde}{\widetilde{\Diff}}

\newcommand{\Scontact}{S_{\contact}}
\newcommand{\Sop}{\Scontact}
\newcommand{\asty}{\star}

\newcommand{\Ssemiop}{\widetilde{\Scontact}}
\newcommand{\Ssemiopstar}{\widetilde{\Scontact^{\asty}}}

\newcommand{\contactlapinv}{{(\widetilde{\Laplaciancontact})}^{_{-1}}}
\newcommand{\llangle}{\langle\!\langle}
\newcommand{\rrangle}{\rangle\!\rangle}

\newcommand{\Laplaciancontact}{\Laplacian_{\contact}}

\newcommand{\Diffsemicontact}{\widetilde{\Diff_{\contact}}}
\newcommand{\Diffssemicontact}{\widetilde{\Diff^s_{\contact}}}

\DeclareMathOperator{\Jac}{Jac}
\newcommand{\mingeodesic}{\chi}
\newcommand{\vorticity}{\omega}

\newtheorem{theorem}{Theorem}[section]
\newtheorem{proposition}[theorem]{Proposition}
\newtheorem{lemma}[theorem]{Lemma}
\newtheorem{corollary}[theorem]{Corollary}

\theoremstyle{definition}
\newtheorem{defn}[theorem]{Definition}
\newtheorem{exmp}[theorem]{Example}

\theoremstyle{remark}

\begin{document}

\selectlanguage{USenglish}

\title[]{Riemannian geometry of the contactomorphism group}
%\author{David G. Ebin and Stephen C. Preston}

\author{David G. Ebin}
\address{Department of Mathematics, Stony Brook University, Stony Brook, NY 11794-3651}
\email{ebin@math.sunysb.edu}
\author{Stephen C. Preston}
\address{Department of Mathematics, University of Colorado,
Boulder, CO 80309-0395} \email{stephen.preston@colorado.edu}
\date{\formatdate{31}{8}{2014}}

\maketitle

\tableofcontents

\section{Introduction}

The ``classical'' diffeomorphism groups of a manifold~\cite{banyaga} are those groups that preserve a volume form, a symplectic form, a contact form, or a contact structure. A Riemannian metric on the manifold generates a right-invariant Riemannian metric on the diffeomorphism group, and the geodesic equation of this metric can be written in terms of what is known as the Euler-Arnold equation~\cite{arnoldkhesin} on its Lie algebra.  This equation can be expressed as a partial differential equation on the manifold; the best-known and most important example is the Euler equation of ideal incompressible fluid mechanics, which is the Euler-Arnold equation on the group of volume-preserving diffeomorphisms. The corresponding equation on the group of symplectomorphisms has been studied in \cite{ebinsymplectic} and \cite{khesinsymplectic}; the Euler-Arnold equation in that case coincides %has many similarities
 with two-dimensional hydrodynamics. In this paper we extend these ideas to study the Euler-Arnold equation on the group of contactomorphisms.

Recall that a contact structure on an orientable manifold $M$ of odd dimension $2n+1$ is the nullspace $\Null(\theta)$ of some %contact
$1$-form $\contact$ which satisfies the nondegeneracy condition that $\contact \wedge d\contact^n$ is nowhere zero. Such a $1$-form is called a contact form.  We will assume that $M$ is equipped with a Riemannian metric which is \emph{associated} to the contact form~\cite{blair} (see Definition \ref{associateddef} below), which simplifies our computations, but all of the results are valid regardless of the metric. Let $\Diff(M)$ be the diffeomorphism group of $M$.
 Then $\eta \in \Diff(M)$ is called a \emph{contactomorphism} if $\eta^*\contact$ is a positive multiple of $\contact$, and we denote the group of such contactomorphisms by $\Diffcontact(M)$. Keeping track of this multiple, we get the group of ``padded contactomorphisms'' $\Diffsemicontact(M) = \{ (\eta, \Lambda) \, \vert \, \eta^*\contact = e^{\Lambda}\contact\}$, a subgroup of the semidirect product $\Diff(M)\ltimes C^{\infty}(M)$ whose group law is given by
\begin{equation}\label{grouplaw}
(\eta, \Lambda)\star (\xi, \Phi) = (\eta\circ\xi, \Lambda\circ\xi + \Phi).
\end{equation}
This subgroup will be our main object of interest.

Its Lie algebra may be identified with the space of smooth functions on $M$ (see Proposition \ref{Scontactdef}), and the Euler-Arnold equation takes the form
\begin{equation}\label{main}
m_t + u(m) + (n+2) m E(f) = 0,
\end{equation}
where $E$ is the Reeb field (defined below), $u=\Scontact f$ is the contact vector field generated by $f$, and $m$ is the momentum given by $m=f-\Laplacian f$ if the Riemannian metric is associated. Equation \eqref{main} reduces to the Camassa-Holm equation~\cite{camassaholm}
\begin{equation}\label{camassaholm}
m_t + fm_{\onedparam} + 2 m f_{\onedparam} = 0, \qquad m = f-f_{\onedparam\onedparam},
\end{equation}
if $M$ is one-dimensional; it is well-known~\cite{kouranbaeva, misiolekCH} that equation \eqref{camassaholm} is the Euler-Arnold equation on $\Diff(S^1)$ with the right-invariant $H^1$ metric. We will see that equation \eqref{main} has many properties in common with both the Camassa-Holm equation and with two-dimensional hydrodynamics.

Our main results are as follows. First we show that \eqref{main} can be expressed as a smooth ordinary differential equation on the Hilbert manifold $\Diffssemicontact(M)$ of contactomorphisms of Sobolev class $H^s$, when $s>n+3/2$. (Recall that $M$ has dimension $2n+1$.) Therefore we have a smooth Riemannian exponential map which is defined in some neighborhood of zero in $T_{\id}\Diffsemicontact(M).$ %in for possibly short time.
As a consequence we have local well-posedness for \eqref{main}: for any $f_0\in H^{s+1}(M)$ there is a unique solution $f(t)\in H^{s+1}(M)$ defined for $t\in (-\varepsilon, \varepsilon)$ of \eqref{main} with $f(0)=f_0$.

We derive a Beale-Kato-Majda type of global existence criterion for solutions of \eqref{main} which says that a solution exists up to time $T$ if and only if the integral $$\int_0^T \lVert \Reeb(f)(t)\rVert_{L^{\infty}} \, dt$$
is finite. One special case occurs if the metric is associated and the Reeb field $\Reeb$ is a Killing field with all orbits closed and of the same length---then we call the metric and contact structure $K$-contact and the contact form %is called
regular. Under these circumstances solutions of \eqref{main} preserve the property that $\Reeb(f)=0$ if it is satisfied initially. Such solutions represent geodesics on the group $\Diffexcontact(M)$ of quantomorphisms, consisting of those diffeomorphisms which preserve the contact form exactly (i.e., $\eta^*\contact = \contact$). We show that this is a totally geodesic subgroup for which all geodesics exist globally in time. An alternative view of the quantomorphism group is as a central extension of the group of Hamiltonian diffeomorphisms of the symplectic manifold $N$ which is obtained as a Boothby-Wang quotient of $M$. For this situation the Euler-Arnold equation takes the form $m_t + \{f,m\} = 0$ with $m=f-\Laplacian f$ on the quotient $N$. This equation is related to the beta-plane approximation for the quasigeostrophic equation in geophysical fluid dynamics, as we shall explain.

Finally we discuss two aspects of equation \eqref{main} which are related to the Camassa-Holm equation. The first is ``peakons,'' singular solutions of \eqref{main} for which the momentum $m$ is initially supported on a set of codimension at least one. The most interesting situation in contact geometry is the case $n=1$ (where $M$ has dimension three) and we consider $m_0$ supported on some surface. Then $m(t)$ is supported on a surface $\Gamma(t)$, and we can write an evolution equation for $\Gamma(t)$. This notion has potential application for the study of overtwisted contact structures; see for example \cite{EKM}.

We conclude by proving some conservation laws. It is well-known that the Camassa-Holm equation is a bihamiltonian equation which is thus completely integrable and has infinitely many conservation laws. The three simplest are $C_{-1} = \int_{S^1} \sqrt{m_+} \,dx$, $C_0 = \int_{S^1} m \, dx$, and $C_1 = \int_{S^1} fm\,d\onedparam$, where $m_+$ is the positive part of the momentum $m=f-f_{\onedparam\onedparam}$.
We show that equation \eqref{main} has the same three conservation laws, although we do not know if any of the other laws generalize or if there is a bihamiltonian structure.

We thank Roberto Camassa, Daniel Fusca, Fran\c cois Gay-Balmaz, Helmut Hofer, Darryl Holm, Boris Khesin, Gerard Misio{\l}ek, Alejandro Sarria, and Cornelia Vizman for many useful discussions during the preparation of this manuscript. The second author was supported by NSF grant DMS-1105660.

\section{Basic Constructs}

\subsection{Contact structures and contact forms}

A contact manifold $(M,\contact)$ is an orientable manifold $M$ of odd dimension $2n+1$ together with a $1$-form $\contact$ such that $\contact \wedge d\contact^n$ is nowhere zero. The contact structure is a distribution in $TM$ defined at each point $p$ as $\Null(\contact)$, the nullspace of $\contact:T_pM \rightarrow {\bf R}$. (In the nonorientable case there are contact structures not determined by contact forms, but for simplicity we do not consider these.) In contact geometry one is primarily concerned with the contact structure~\cite{geiges, EKM}, and thus a contact form $\contact$ is equivalent to $F\contact$ whenever $F$ is a positive function on $M$. A diffeomorphism $\eta\in \Diff(M)$ is called a \emph{contactomorphism} if $\eta^*\contact = e^{\Lambda} \contact$ for some function $\Lambda\colon M\to \mathbb{R}$. In some cases one is concerned with the contact form itself, and we say that a \emph{quantomorphism} is a diffeomorphism $\eta$ such that $\eta^*\contact=\contact$; see \cite{ratiuschmid} and Section \ref{quantosection} below for details. The Riemannian geometry of the group of quantomorphisms was studied by Smolentsev~\cite{smolentsevquanto}, but to our knowledge the Riemannian geometry of the group of contactomorphisms has never been studied in depth.

Our primary concern is with the Lie algebra of contact vector fields, those for which the local flow preserves the contact structure. We review some of the basic concepts; see Geiges~\cite{geiges} for more details. Given a contact form $\contact$, there is a unique vector field $\Reeb$, called the \emph{Reeb field}, defined by the conditions
$$\iota_{\Reeb}d\contact = 0 \quad\text{and}\quad \contact(\Reeb)\equiv 1.$$  The uniqueness of $\Reeb$ is a direct consequence of the fact that $\contact \wedge d\contact^n$ is never zero. The following characterization of contact vector fields is well-known.

\begin{proposition}\label{Scontactdef}
The Lie algebra $T_{\id}\Diffcontact(M)$ consists of vector fields $u$ such that
 $\Lie_u\contact = \lambda \contact$ for some function $\lambda\colon M\to\mathbb{R}$. Any such field $u$ is uniquely determined by the function $f = \contact(u)$,
so we write  $u=\Scontact f$. In this case the multiplier $\lambda$ is given by $\Reeb(f)$.
The padded contactomorphism group $\Diffsemicontact(M)$ has Lie algebra of the form $$T_{\id}\Diffsemicontact(M) = \big\{ \Ssemiop f = (\Scontact f, \Reeb f) \, \vert \, f\in C^{\infty}(M)\big\}.$$
\end{proposition}

\begin{proof}
Given a family of contactomorphisms $\eta(t)$ satisfying $\eta(t)^*\contact = e^{\Lambda(t)} \contact$ with $\eta(0)=\id$ and $\dot{\eta}(0) = u$, differentiating at $t=0$ gives $\Lie_u\contact = \dot{\Lambda}(0) \contact$; conversely
given any vector field $u$ such that $\Lie_u\contact = \lambda \contact,$ then $\eta(t)$, the flow of $u$, satisfies $\eta(t)^*\contact = e^{t\lambda} \contact$.
Hence $u\in T_{\id}\Diffcontact(M)$ iff $\Lie_u\contact = \lambda \contact$ for some function $\lambda$.

Since the Reeb field $\Reeb$ is not in the nullspace $\Null(\contact)$,  any vector field $u$ can be decomposed as $u=v+f\Reeb$ where $v \in \Null(\contact).$
If $u\in T_{\id}\Diffcontact(M)$ is decomposed as above, then $f=\contact(u)$, and by the Cartan formula we have $\iota_ud\contact + df = \lambda \contact$.  Applying both sides to the Reeb field $\Reeb$ we obtain $\lambda = \Reeb(f)$. The fact that  $\contact \wedge d\contact^n$ is never zero implies that $d\contact$ must have rank $2n$ at each point. Hence the map $u\mapsto \iota_ud\contact$ is an isomorphism in each tangent space $T_pM$ from $\Null( \contact) \subset T_pM$ to $\Ann(\Reeb)\subset T_p^*M$ (the annihilator of $\Reeb$). We denote this map by $\gamma$ and its inverse by $\Gamma$; then $\Lie_u\contact = \lambda \contact$ if and only if $\lambda = \Reeb(f)$ and $u = \Gamma\big( \Reeb(f)\contact - df\big) + f \Reeb$, where $f=\contact(u)$.
Define $\Scontact(f)$ to be $u$.
\end{proof}

By the Darboux theorem~\cite{geiges}, every contact form can be expressed in some local coordinates $(x^1,\cdots, x^n, y^1,\cdots, y^n, z)$ as $\theta = dz - \sum_{k=1}^n y^k \,dx^k$. In such coordinates the Reeb field is given by $\Reeb = \frac{\partial}{\partial z}$, and the operator $S_{\contact}f$ is given in terms of the frame $P_k = \frac{\partial}{\partial x^k} + y^k \, \frac{\partial}{\partial z},$ $Q_k = \frac{\partial}{\partial y^k}$ and $E$ as
\begin{equation}\label{Scontactdarboux}
S_{\contact}f = -\sum_k Q_k(f) P_k + \sum_k P_k(f) Q_k + f \Reeb.
\end{equation}
We note that $\Sop f$ differentiates $f$ in only $2n$ directions; the omitted direction is $\Reeb$. This will be important later when we discuss smoothness in the Sobolev context.

The padded contactomorphism group $\Diffsemicontact(M) = \{(\eta, \Lambda) \, \vert \, \eta^*\contact = e^{\Lambda} \contact\}$ has the structure of a semidirect product, since if $\eta^*\contact = e^{\Lambda} \contact$ and $\xi^*\contact = e^{\Phi}\contact$, then $(\eta\circ\xi)^*\contact = \xi^*(e^{\Lambda}\contact) = e^{\Lambda\circ\xi+\Phi} \contact$. Hence the group law is as given by \eqref{grouplaw}.
%\begin{equation}\label{grouplaw}
%(\eta, \Lambda)\star (\xi, \Phi) = (\eta\circ\xi, \Lambda\circ\xi + \Phi).
%\end{equation}
%The Lie algebra is
%$$ T_{\id}\Diffsemicontact(M) = \{ (u, \lambda)\in T_{\id}\Diff(M)\ltimes C^{\infty}(M) \, \vert \, \Lie_u\contact = \lambda \contact\};$$
%we write $(u,\lambda) = \widetilde{S_{\contact}}f = (S_{\contact}f, E(f))$.
The Lie bracket is given by
\begin{equation}\label{contactpoisson}
[\Ssemiop f, \Ssemiop g] = \Ssemiop \{f,g\}, \qquad \text{where } \{f,g\} = \Sop f(g) - g \Reeb(f);
\end{equation}
we will refer to the bracket on functions as a ``contact Poisson bracket.''
Note that unlike a symplectic bracket it does not satisfy the Leibniz rule.

\subsection{Associated Riemannian metrics}

We now want to consider a Riemannian structure on $M$. Although in principle the analysis is very similar whether or not the Riemannian metric is related to the contact form in any way, it is convenient to require some stronger compatibility. A reasonable minimum condition is that the volume form generated by the Riemannian metric be a constant multiple of $\contact \wedge (d\contact)^n$, which ensures that the Reeb field $\Reeb$ is divergence-free. A stronger condition is that the Riemannian metric be \emph{associated} to the contact form.

\begin{defn}\label{associateddef}
If $M$ is a contact manifold with contact form $\contact$ and Reeb field $\Reeb$, a Riemannian metric $\langle \cdot, \cdot\rangle$ is called \emph{associated} if it satisfies the following conditions:
\begin{itemize}
\item $\contact(u) = \langle u, \Reeb\rangle$ for all $u\in TM$, and
\item there exists a $(1,1)$-tensor field $\phi$ such that $\phi^2(u) = -u + \contact(u) \Reeb$ and $d\contact(u,v) = \langle u, \phi v\rangle$ for all $u$ and $v$.
\end{itemize}
\end{defn}

It is known that every contact manifold has an infinite-dimensional family of associated Riemannian metrics~\cite{blair}.

Having an associated metric allows us to simplify some formulas, as follows.

\begin{proposition}\label{associatedprop}
Suppose $M$ has a contact form $\contact$ and an associated Riemannian metric as in Definition \ref{associateddef}. Then we have the following:
\begin{itemize}
\item $\Reeb$ is a unit vector field.
\item Contact vector fields are given by
\begin{equation}\label{associatedScontact}
\Sop f = f \Reeb - \phi \grad f
\end{equation}
\item The momentum $m = \Ssemiopstar \Ssemiop f$, where $\Ssemiopstar$ is the formal adjoint of $\Ssemiop$, is given by
\begin{equation}\label{associatedmomentum}
m = f - \Laplacian f.
\end{equation}
\item There is an orthonormal frame $\{\Reeb, P_1,\ldots, P_n, Q_1,\ldots, Q_n\}$ such that $\phi(P_k) = -Q_k$, $\phi(Q_k) = P_k$, and $\phi(\Reeb)=0$. Thus we have
    $\Sop f = f \Reeb +\sum_{k=1}^n P_k(f) Q_k - Q_k(f) P_k$,
    as in \eqref{Scontactdarboux}.
\item If $n\ge 1$, the Riemannian volume form $\mu$ is given by $\mu = \frac{1}{n} \contact \wedge (d\contact)^n$.
\item For any function $f$, we have
\begin{equation}\label{divergenceSop}
\diver{(\Sop f)} = (n+1) \Reeb(f).
\end{equation}
In particular since $\Reeb = \Sop(1)$, the Reeb field $\Reeb$ is divergence-free.
\end{itemize}
\end{proposition}

\begin{proof}
$\Reeb$ is unit since $\lvert \Reeb\rvert^2 = \contact(\Reeb) = 1$. Since $\contact(\phi(v)) = \langle \Reeb, \phi(v)\rangle = d\contact(\Reeb, v) = 0$ for any vector $v$, we see that $\phi$ maps into the nullspace of $\contact$. Hence if $u=f\Reeb - \phi\grad f$ we will have (for any vector field $v$)
%red
\begin{align*}
\Lie_u\contact(v) &= d\contact(u,v) + v\big( \contact(u)\big) \\
&= d\contact(f\Reeb - \phi\grad f, v) + v(f) \\
&= d\contact(v, \phi \grad f) + v(f) \\
&= \langle v, \phi^2 \grad f\rangle + \langle v, \grad f\rangle \\
&= -\langle v, \grad f\rangle + \langle v, \contact(\grad f) \Reeb\rangle + \langle v,\grad f\rangle \\
&= \Reeb(f) \contact(v).
\end{align*}
Since this is true for any $v$ we must have $\Lie_u\contact = \Reeb(f) \contact$ which means that $u = \Sop f$.

To obtain the formula for the momentum, we need to compute the formal adjoint of
$\Ssemiop.$ Integrating by parts we get:
%red
\begin{align*}
\llangle \Ssemiop f, \Ssemiop g\rrangle &= \int_M \langle \Sop f, \Sop g\rangle \, d\mu + \int_M \Reeb(f)\Reeb(g) \, d\mu \\
%&= \int_M \langle f\Reeb - \phi \grad f, g \Reeb - \phi \grad g\rangle \, d\mu + \int_M \Reeb(f) \Reeb(g) \, d\mu \\
&= \int_M \big( fg + \langle \phi\grad f, \phi \grad g\rangle + \Reeb(f)\Reeb(g)\big) \ d\mu,
\end{align*}
since $\Reeb$ is orthogonal to the image of $\phi$. Now \begin{multline*}
\langle \phi\grad f,\phi \grad g\rangle = d\contact(\phi \grad f, \grad g) = -d\contact(\grad g, \phi \grad f) = -\langle \grad g, \phi^2 \grad f\rangle \\
= \langle \grad g, \grad f\rangle - \langle \grad g, \contact(\grad f) \Reeb\rangle = \langle \grad g, \grad f\rangle - \Reeb(f)\Reeb(g).
\end{multline*}
We conclude that
 \begin{align*}
 \int_M g \Ssemiopstar\Ssemiop f\, d\mu &= \llangle \Ssemiop f, \Ssemiop g\rrangle = \int_M \big( fg + \langle \grad f, \grad g\rangle \big) \, d\mu \\
 &= \int_M g(f - \Laplacian f) \, d\mu
 \end{align*}
for any function $g$, as desired.

The orthonormal basis is constructed as follows: take an arbitrary unit vector field $Q_1$ orthogonal to $\Reeb$, and define $P_1 = \phi(Q_1)$; then $\langle P_1, Q_1\rangle = \langle Q_1, \phi(Q_1)\rangle = d\contact(Q_1,Q_1)=0$, and $P_1$ is orthogonal to $\Reeb$ since it is in the image of $\phi$. Clearly $\phi(P_1) = -Q_1$. Choose $Q_2$ orthogonal to all three, and $P_2 = \phi(Q_2)$; then $\langle P_2, Q_1\rangle = \langle \phi Q_2, Q_1\rangle = -\langle Q_2, P_1\rangle = 0$ and $\langle P_2, P_1\rangle = \langle \phi Q_2, \phi Q_1\rangle = -\langle Q_2, Q_1\rangle = 0$ since $\phi^2$ is the negative identity on the orthogonal complement of $\Reeb$. We continue in this way to obtain the orthonormal frame, then use the fact that $\grad f = \Reeb(f) \Reeb + \sum_k P_k(f) P_k + Q_k(f) Q_k$ to obtain the formula for $\Sop f$ in the basis.

To compute the Riemannian volume form $\mu$, we note that since the basis is orthonormal, we have $\mu(\Reeb, P_1, \ldots, P_n, Q_1,\ldots, Q_n) = \pm 1$.
Now let $\nu = \contact\wedge (d\contact)^n$; then we need to compute $\nu(\Reeb, P_1,\ldots, P_n, Q_1,\ldots, Q_n)$.
Since $\contact(\Reeb)=1$ and $\contact(P_k)=\contact(Q_k)=0$, we have
$$\nu(\Reeb, P_1,\ldots, P_n, Q_1,\ldots, Q_n) = (d\contact)^n(P_1,\ldots, P_n,Q_1,\ldots, Q_n).$$
In addition we have $d\contact(P_k,Q_k) = \langle P_k, \phi Q_k\rangle = 1$ for each $k$, so that
$$
(d\contact)^n(P_1,\ldots, P_n, Q_1,\ldots, Q_n) = n d\contact(P_1,Q_1)\cdots d\contact(P_n,Q_n) = n.$$
 Hence $\nu = n \mu$.

Finally to obtain the divergence of $u=\Sop f$, we note that $\Lie_u\contact = \Reeb(f) \contact$; hence we have $\Lie_ud\contact = d\big(\Reeb(f)\contact\big) = d\Reeb(f)\wedge \contact + \Reeb(f) \, d\contact$. By the product rule for Lie derivatives we obtain $\Lie_u(\contact \wedge d\contact) = 2\Reeb(f) \contact \wedge d\contact$, and inductively we get $\Lie_u\nu = (n+1)\Reeb(u) \nu$. In particular since $\Reeb=\Sop(1)$ and $\Reeb(1)=0$, the Reeb field is divergence-free.
\end{proof}

%red
Here are some examples of manifolds with associated metrics.

\begin{exmp}\label{contactexamples}
\mbox{}
\begin{itemize}

\item On $S^1$ with coordinate $\onedparam$ and the basic $1$-form $\contact = d\onedparam$, we have $\Reeb = \frac{\partial}{\partial \onedparam}$ and $\Sop f = f \Reeb$ with $\contactlap f = f-f_{\onedparam\onedparam}$. The group of contactomorphisms is of course all of $\Diff(S^1)$.
\item On any three-dimensional unimodular Lie group~\cite{milnor} with a frame of left-invariant vector fields $\{e_1, e_2, e_3\}$ satisfying $[e_2, e_3] = -e_1$,
denote the dual frame by $\{e_1^{\flat}, e_2^{\flat}, e_3^{\flat}\}$ and let $\alpha = e_1^{\flat}.$ Then the $1$-form $\alpha$ is a contact form since $\alpha\wedge d\alpha(e_1,e_2,e_3) = -\alpha(e_1) \alpha([e_2,e_3]) = 1$. Declaring these fields to be orthonormal, we get an associated metric where $\Sop f = fe_1 -e_3(f) e_2 + e_2(f) e_3$ and $\contactlap = 1 - e_1^2 - e_2^2 - e_3^2$. The $3$-sphere and the Heisenberg group are special cases; in particular on the Heisenberg group the Darboux contact form $dz - y\,dx$ has associated metric $ds^2 = (dz-y\,dx)^2 + dx^2 + dy^2$.
\item On $\mathbb{T}^3$ the $1$-form $\contact = \sin{z} \, dx + \cos{z} \, dy$ is a contact form with the usual flat metric associated, such that $\Sop f =
(f\sin{z} + f_z \cos{z}) \, \partial_x + (f\cos{z} - f_z \sin{z}) \, \partial_y + (-f_x \cos{z} + f_y \sin{z}) \, \partial_z$ and $\contactlap = 1-\partial_x^2 - \partial_y^2 - \partial_z^2$.
\end{itemize}
\end{exmp}

\subsection{The Riemannian structure of the contactomorphism group}

As is usual when studying diffeomorphism groups~\cite{ebinmarsden, misiolekpreston}, the Fr\'echet manifold structure leads to analytical difficulties when studying geometry due to the lack of an Inverse Function Theorem
and to the possibility of non-integrability of vector fields. Hence we enlarge the group to the set of diffeomorphisms $\eta\in \Diff^s(M)$ of Sobolev class $H^s$ for $s>\dim{M}/2+1=n+3/2$, consisting of those maps whose derivatives up to order $s$ are square-integrable in every coordinate chart of compact support. The index $s$ is large enough to ensure by the Sobolev embedding theorem that $\eta$ and $\eta^{-1}$ are both $C^1$. We denote by $\Diffscontact(M)$ the group of Sobolev $H^s$ diffeomorphisms preserving the contact structure; although this is a subgroup of $\Diff^s(M)$, it is \emph{not} a smooth submanifold since $T_{\id}\Diffscontact(M)$ is not a closed subspace of $T_{\id}\Diffs(M)$. (See \cite{omori} and \cite{smolentsev}.) The problem is that $\Sop f$ (as given in coordinates by \eqref{Scontactdarboux}) does not differentiate the function $f$ in all directions: the derivative in the Reeb direction is missing.

Following Omori~\cite{omori}, we resolve this by instead considering $\Diffssemicontact(M) = \{ (\eta, \Lambda)\in \Diffs(M)\ltimes H^s(M)\, \vert \, \eta^*\contact = e^{\Lambda} \contact\}$ as a subgroup of $\Diffssemi(M) = \Diffs(M)\ltimes H^s(M)$. Note that requiring $\Lambda\in H^s(M)$ is not the obvious definition: $\Lambda$ has more smoothness than would be expected automatically since $\eta^*\contact$ is usually only $H^{s-1}$ if $\eta\in \Diffs(M)$.
However it is easy to check that for any $f\in H^{s+1}(M)$, the flow $(\eta(t), \Lambda(t))$ of the vector field $\Ssemiop f$ on $M\times \mathbb{R}$ will satisfy $(\eta(t),\Lambda(t))\in \Diffssemicontact(M)$ for all $t$.
Omori shows (in our notation) that the map $\Psi \colon \Difftilde^s(M) \to H^{s-1}(\Omega^1) \oplus H^{s-1}(\Omega^2)$
defined by
$$ \Psi(\eta, \lambda) = \big(e^{-\Lambda}\eta^*\contact, e^{-\Lambda}(-d\Lambda \wedge \eta^*\contact + \eta^*d\contact)\big)$$
is smooth and that $\Diffssemicontact(M)$ is the inverse image of the regular value $(\contact, d\contact)$, so that it is a smooth submanifold. Because of this, we will work primarily with the padded contactomorphism group, so that our geodesic equation ends up being a smooth ordinary differential equation on a Hilbert manifold.

Another approach to the contactomorphism group appears in Bland-Duchamp~\cite{BD}; they use the Folland-Stein~\cite{FS} topology rather than the usual Sobolev topology, and in this topology the contactomorphism group is a smooth Hilbert submanifold of the diffeomorphism group. The main reason we prefer the present approach is that our geodesic equation ends up having the momentum $m$ in \eqref{main} defined in terms of an elliptic operator rather than a subelliptic operator, and the one-dimensional equation reduces to the Camassa-Holm equation \eqref{camassaholm} which is a smooth ODE rather than
$u_t+3uu_x = 0$ which is not a smooth ODE~\cite{constantinkolev}.

Given a right-invariant Riemannian metric on any Lie group $G$, the geodesic equation for a curve $\eta\colon (a,b)\to G$ may be written generally as
\begin{equation}\label{eulerarnold}
\frac{d\eta}{dt} = dR_{\eta(t)} u(t), \qquad \frac{du}{dt} + \ad_u^*u = 0
\end{equation}
where the second equation is called the Euler-Arnold equation. See for example \cite{KLMP2} for a survey of such equations. The most famous examples are the Euler equations for an ideal fluid if $G$ is the group of volume-preserving diffeomorphisms, and the Korteweg-deVries and Camassa-Holm equations when $G$ is $\Diff(S^1)$ or its central extension.
In our case, the simplest right-invariant Riemannian metric on the semidirect product $\Diffsemi(M)$ is given at the identity by
\begin{equation}\label{rightinvmetric}
\llangle (u, \lambda), (v, \rho)\rrangle = \int_M \langle u, v\rangle \, d\mu + \int_M \lambda \rho \, d\mu.
\end{equation}
The Euler-Arnold equation on a semidirect product has been studied in \cite{epdiff} and \cite{vizman} in general, and in special cases such as the ``two-component generalizations'' of some well-known one-dimensional Euler-Arnold equations; see for example \cite{lenellswunsch} and references therein.

The metric \eqref{rightinvmetric} induces a right-invariant metric on the submanifold $\Diffsemicontact(M)$ which at the identity takes the form
\begin{equation}\label{contactometric}
\llangle
\Ssemiop f, \Ssemiop g
\rrangle = \int_M \langle \Sop f, \Sop g\rangle \, d\mu + \int_M \Reeb(f) \Reeb(g)\, d\mu = \int_M (f-\Laplacian f) g \, d\mu,
\end{equation}
as computed in Proposition \ref{associatedprop} for an associated metric. This metric of course
gives the same topology as the Sobolev $H^1$ metric on functions.
More generally (if the Riemannian metric is not associated), the metric \eqref{contactometric} becomes
$$ \llangle \Ssemiop f, \Ssemiop g\rrangle = \int_M m g \, d\mu,$$
where $m=\Ssemiopstar\Ssemiop f$ is the contact Laplacian (a positive-definite elliptic operator), and the metric induced on functions is
topologically equivalent
to the Sobolev $H^1$ metric.

We now compute the Euler-Arnold equation on the contactomorphism group.

\begin{proposition}\label{derivationprop}
Suppose $M$ is a contact manifold with an associated Riemannian metric. Then the Euler-Arnold equation \eqref{eulerarnold} on $\Diffsemicontact(M)$ with right-invariant metric \eqref{contactometric} is given by \eqref{main}, where $m = \contactlap f=f-\Laplacian f$.
\end{proposition}

\begin{proof}
We just need to compute $\ad_{\tilde{u}}^*\tilde{u}$, where $\tilde{u} = \Ssemiop f$ for some stream function $f$. Let $\tilde{v} = \Ssemiop g$; then we have
\begin{multline*}
\llangle \ad_{\tilde{u}}^*\tilde{u}, \tilde{v}\rrangle = \llangle \tilde{u}, \ad_{\tilde{u}}\tilde{v}\rrangle = -\llangle \Ssemiop f, \Ssemiop\{f,g\} \rrangle = -\int_M \{f,g\} \contactlap f \, d\mu \\
= -\int_M \big[ \Sop f(g) - g\Reeb(f)\big] m \, d\mu = \int_M g\big[\diver{\big( m \Sop f\big)} + m \Reeb(f)\big] \, d\mu,
\end{multline*}
using the formula \eqref{contactpoisson} and the fact that the Lie algebra adjoint is the negative of the usual Lie bracket of vector fields.

%It remains to compute the divergence of $\Sop f$. By Proposition \ref{associatedprop}, if $\nu = \contact \wedge (d\contact)^n$ then we have
%$$ \Lie_{\Sop f} \nu = (\diver{\Sop f}) \, \nu.$$
%Using the product rule for Lie derivatives and the fact that $\Lie_{\Sop f}\contact = \Reeb(f) \contact$, it is easy to compute that
%\begin{equation}\label{divergencecontact}
%\diver{(\Sop f)} = (n+1) \Reeb(f).
%\end{equation}
By Proposition \ref{associatedprop} we have $\diver{(\Sop f)} = (n+1) \Reeb(f)$. Hence we have
$$ \llangle \ad_{\tilde{u}}^*\tilde{u}, \tilde{v}\rrangle = \int_M g \big( \Sop f(m) + (n+2) m\Reeb(f) \big) \, d\mu,$$
where $m = \contactlap f$.
We obtain $\ad_{\tilde{u}}^*\tilde{u} = \Ssemiop \contactlapinv\big[ \Sop f(m) + (n+2) m \Reeb(f) \big]$, and applying $\Ssemiopstar$ to the second equation of \eqref{eulerarnold}, we obtain \eqref{main}.
\end{proof}

Every Euler-Arnold equation on a group $G$ has a conservation law (which reflects the symmetry obtained by the Noether theorem~\cite{AMR} resulting from right-invariance of the metric). In general this comes from rewriting \eqref{eulerarnold} as
$ \frac{d}{dt} \Ad_{\eta(t)}^* u(t) = 0$ to obtain
\begin{equation}\label{vorticitylawgeneral}
u(t) = \Ad_{\eta(t)^{-1}}^*u_0.
\end{equation}
This with the flow equation leads to a first-order equation on the group $G$ given by
\begin{equation}\label{generalconservation}
\frac{d\eta}{dt} = dR_{\eta(t)} \Ad_{\eta(t)^{-1}}^* u_0, \qquad \eta(0)=\id,
\end{equation}
where $u_0\in T_eG$ is the initial velocity. For ideal fluid mechanics, equation \eqref{vorticitylawgeneral} expresses conservation of vorticity; for the Camassa-Holm equation \eqref{camassaholm} it expresses the conservation of the momentum $m=f-f_{\onedparam\onedparam}$ in the form
\begin{equation}\label{camassaconservation}
 m(t,\eta(t,\onedparam)) = m_0(\onedparam)/\eta_{\onedparam}(t,\onedparam)^2,
\end{equation}
where $m_0\colon S^1\to\mathbb{R}$ is the initial momentum and $\eta(t)\in \Diff(S^1)$ is the Lagrangian flow.
In particular if $m_0$ is of one sign, then $m(t)$ is always of the same sign. A well-known result due to McKean~\cite{mckean} is that the Camassa-Holm equation on the circle has global solutions if and only if the momentum never changes sign; if it does change sign, solutions $u$ blow up in finite time due to $\eta$ ceasing to be a diffeomorphism. The following lemma relating the Jacobian determinant of $\eta$ to the scaling factor $\Lambda$ will be useful.

\begin{lemma}\label{jacobianlemma}
Suppose the Riemannian volume form $\mu$ is a constant multiple of the contact volume form $\nu = \contact \wedge (d\contact)^n$, as for example happens when the Riemannian metric is associated (Proposition \ref{associatedprop}). Then whenever $\eta^*\contact = e^{\Lambda}\contact$, the Jacobian determinant $\Jac(\eta)$ defined by $\eta^*\mu = \Jac(\eta)\mu$ will satisfy
\begin{equation}\label{jacobiancontact}
\Jac(\eta) = e^{(n+1)\Lambda}.
\end{equation}
\end{lemma}

\begin{proof}
Since $\eta^*\contact = e^{\Lambda}\contact$, we have $\eta^*(d\contact) = d(\eta^*\contact) = e^{\Lambda}(d\contact + d\Lambda \wedge \contact)$. Now we have
$$
\eta^*(\contact \wedge d\contact) = (\eta^*\contact) \wedge (\eta^*d\contact) = e^{2\Lambda} (\contact\wedge d\contact - \contact\wedge \contact \wedge d\Lambda) = e^{2\Lambda} \contact \wedge d\contact.$$
Inductively we obtain $\eta^*\nu = e^{(n+1)\Lambda} \nu$, and since $\mu$ is a constant multiple of $\nu$ we have $\Jac(\eta) = e^{(n+1)\Lambda}$. Alternatively this is a consequence of the divergence formula \eqref{divergenceSop}.
\end{proof}

Now we generalize the conservation law \eqref{camassaconservation} to the higher-dimensional situation.

\begin{proposition}\label{conservationlawprop}
Let $M$ be a contact manifold of dimension $2n+1$ with an associated Riemannian metric as in Definition \ref{associateddef}. Then equation \eqref{main} satisfies the conservation law
\begin{equation}\label{conservation}
m\big(t, \eta(t,p)\big) = m_0(p)/\Jac(\eta(t,p))^{(n+2)/(n+1)},
\end{equation}
where $\eta$ is the Lagrangian flow of $u=\Sop f$ and $\Jac(\eta)$ denotes its Jacobian determinant.
\end{proposition}

\begin{proof}
For any $(\eta,\Lambda)\in \Diffsemicontact(M)$ and $(v,\phi)\in T_{\id}\Diffsemicontact(M)$, we have by definition that $\Ad_{(\eta, \Lambda)}(v,\phi) = \frac{d}{dt}\big|_{t=0} (\eta,\Lambda)\star \big(\xi(t), \Phi(t)\big) \star (\eta,\Lambda)^{-1}$, where $\big(\xi(t), \Phi(t)\big)$ is any curve satisfying $\frac{d}{dt}\big|_{t=0} \big(\xi(t), \Phi(t)\big) = (v,\phi)$. From the formula \eqref{grouplaw} for the group law, it is easy to compute that
$$ \Ad_{(\eta, \Lambda)}(v,\phi)
%= \frac{d}{dt}\Big|_{t=0} (\eta\circ\xi(t), \Lambda\circ\xi(t) + \Phi(t))\star (\eta^{-1}, -\Lambda\circ\eta^{-1}) =
%\frac{d}{dt}\Big|_{t=0} \big( \eta\circ\xi(t)\circ\eta^{-1}, \Lambda\circ\xi(t)\circ\eta^{-1}+\Phi(t)\circ\eta^{-1}-\Lambda\circ\eta^{-1}\big)
= (\Ad_{\eta}v, v(\Lambda)\circ \eta^{-1} + \phi\circ\eta^{-1}),$$
where $\Ad_{\eta}v$ is the usual adjoint operator on the diffeomorphism group. We now need to compute what this is when $\eta$ is a contactomorphism and $v=\Sop g$ is a contact vector field. Since $\Ad_{\eta}v$ will also be a contact vector field, we must have $\Ad_{\eta}\Sop g = \Sop h$ where $h = \contact(\Ad_{\eta}v)$.
%We may compute this using (at each $p\in M$) that
%$$h(p) = \contact_p(\eta_*v\circ\eta^{-1}(p)) = (\eta^*\contact)\big(v\circ\eta^{-1}(p)\big) = e^{\Lambda(\eta^{-1}(p))} \contact\big( v\circ\eta^{-1}(p)\big)
%= e^{\Lambda(\eta^{-1}(p))} g(\eta^{-1}(p)) = e^{\Lambda}g(\eta^{-1}(p)).$$
Using $\eta^*\contact = e^{\Lambda}\contact$ and the formula $\Ad_{\eta}v = \eta_*v\circ\eta^{-1}$, it is easy to verify that  $h = (e^{\Lambda}g)\circ\eta^{-1}$.

We thus compute
\begin{align*}
\big\llangle \Ad_{(\eta, \Lambda)}^*(u,\lambda), (v,\phi)\big\rrangle
%= \llangle (u,\lambda), \Ad_{(\eta, \Lambda)}(v,\phi)\rrangle
&= \big\llangle \Ssemiop f, \Ssemiop \big(e^{\Lambda\circ\eta^{-1}} (g\circ\eta^{-1})\big)\big\rrangle \\
%= \int_M (\contactlap f) e^{\Lambda\circ\eta^{-1}} g\circ\eta^{-1} \, d\mu
&= \int_M (\contactlap f) e^{\Lambda} g\circ\eta^{-1} \, d\mu \\
&= \int_M (\contactlap f\circ\eta) e^{(n+2) \Lambda} g \, d\mu,
\end{align*}
using the change of variables formula and Lemma \ref{jacobianlemma}. We conclude that
\begin{equation}\label{adstar}
\Ad_{(\eta, \Lambda)}^*\Ssemiop f = \Ssemiop \contactlap^{-1} \Big((m\circ\eta) e^{(n+2) \Lambda}\Big),
\end{equation}
where $m=\contactlap f$. Applying $\Ssemiopstar$ to both sides of $\Ad_{(\eta(t),\Lambda(t))}^*\Ssemiop f(t) = \Ssemiop f_0$ and using \eqref{jacobiancontact}, we obtain \eqref{conservation}.
\end{proof}

Of course, we could also have derived \eqref{conservation} directly by writing the Lagrangian flow in the form
$$ \frac{\partial \eta}{\partial t}(t,p) = \Sop f\big(t, \eta(t,p)\big), \qquad \frac{\partial \Lambda}{\partial t}(t,p) = \Reeb(f)\big(t,\eta(t,p)\big)$$
and composing \eqref{main} with $\eta$ to obtain
$$ \frac{\partial}{\partial t} m\big(t,\eta(t,p)\big) + (n+2) m(t,\eta(t,p)\big) \frac{\partial \Lambda}{\partial t}(t,p) = 0,$$ which immediately integrates to \eqref{conservation}.
But Proposition \ref{conservationlawprop} makes clear the analogy with vorticity and momentum conservation in the general Euler-Arnold equation.

There are two significant features of the conservation law \eqref{conservation}: the first is that the momentum is a function on the manifold rather than a vector field as it is for the EPDiff equation~\cite{epdiff} (another suggested higher-dimensional version of the Camassa-Holm equation). We may thus conjecture that the sign of the momentum controls global existence of solutions as it does for the Camassa-Holm equation; see \cite{prestonsarria} for an analysis of a closely-related case.  The second is that the equation \eqref{generalconservation} can be shown to be a smooth ODE on the Sobolev manifold $\Diffssemicontact(M)$, following the methods of \cite{ebinfirstorder} or Majda-Bertozzi~\cite{majdabertozzi}. Thus we can avoid the somewhat complicated geometric machinery of \cite{ebinmarsden}. We will do this in the next section.

\section{Local and global existence}

We now restrict our attention to the Hilbert manifold $\Diffssemicontact(M)$, which as noted above is a smooth submanifold of $\Diffssemi(M) = \Diffs(M) \ltimes H^s(M)$. Our right-invariant metric \eqref{contactometric} is induced by the right-invariant metric on $\Diffssemi(M)$, and thus we could in principle use the methods of \cite{ebinmarsden} to prove that the tangential projection is smooth and thus that the geodesic equation is a smooth ODE on $T\Diffssemicontact(M)$. However this relies on the fact that the geodesic equation on $\Diffssemi(M)$ is a smooth ODE, which is probably true but is not proven in the literature to our knowledge. We will therefore work directly on $\Diffssemicontact(M)$ using the conservation law \eqref{generalconservation} to write the geodesic equation as a first-order ODE on $\Diffssemicontact(M)$: we obtain
\begin{equation}
\frac{d}{dt} (\eta, \Lambda) %= \Ad_{(\eta, \Lambda)^{-1}^*\Ssemiof f_0
= (\Ssemiop)_{(\eta, \Lambda)} \contactlapinv_{(\eta, \Lambda)} \big(m_0 e^{-(n+2)\Lambda}\big),
\end{equation}
where the ``twisted operators'' are defined as $(\Ssemiop)_{(\eta, \Lambda)} = dR_{(\eta,\Lambda)}\circ \Ssemiop \circ dR_{(\eta,\Lambda)^{-1}}$ and $\contactlapinv_{(\eta,\Lambda)}= dR_{(\eta,\Lambda)}\circ \contactlapinv\circ dR_{(\eta,\Lambda)^{-1}}$. If we could prove that these twisted operators were smooth in $(\eta,\Lambda)\in \Diffssemicontact(M)$, we would be done. However although $(\Ssemiop)_{(\eta,\Lambda)}\colon H^{s+1}(M)\to T_{(\eta,\Lambda)}\Diffssemi(M)$ is smooth (like all twisted first-order differential operators, as in \cite{ebinmarsden}), the operator $\contactlapinv_{(\eta, \Lambda)}$ is not, and in fact does not even map into the correct space. We need it to map from $H^{s-1}(M)$ to $H^{s+1}(M)$, but it cannot map into $H^{s+1}(M)$ since $\eta$ is only $H^s$.
Instead we use the fact that the operator  $(\Ssemiop)_{(\eta, \Lambda)} \contactlapinv_{(\eta, \Lambda)}$ is the inverse of $\Ssemiopstar_{(\eta, \Lambda)}\colon T_{(\eta,\Lambda)}\Diffssemi(M)\to H^{s-1}(M)$, and $\Ssemiopstar_{(\eta, \Lambda)}$ is smooth in $(\eta, \Lambda)$, using a simplified version of the technique from \cite{ebinmarsden}.

%  The alternative approach is to follow Majda-Bertozzi~\cite{majdabertozzi} to get an explicit formula for $\contactlapinv_{(\eta,\Lambda)}$, but a formula for the kernel of this operator is unwieldy on a general Riemannian manifold.

We will assume the Riemannian metric on $M$ is associated to the contact form to simplify the notation, although this assumption is not necessary to prove the theorem. We assume $M$ is compact in order to use the standard results of Sobolev manifolds of maps (as in \cite{ebinmarsden}).

%The essential issue is to prove smoothness of the twisted operator
%$(\Scontactbar \, \Laplaciancontactbar^{-1})_{\eta}$; the problem is that although the twisted operator $(\Scontactbar)_{\eta}$ is smooth (as a differential operator), the operator $(\Laplaciancontactbar^{-1})_{\eta}$ does not even make sense (since it doesn't map into the correct space). We therefore use the fact that this twisted operator is the inverse of $(\Scontactstarbar)_{\eta}$ on a certain restricted space, together with the inverse function theorem; this is essentially the same technique as in \cite{EM} and \cite{EP}.

\begin{theorem}\label{mainthm1}
Let $M$ be a compact contact manifold with $\dim{M}=2n+1$ and let $s$ be an integer with $s>n+\tfrac{3}{2}$. Assume the Riemannian metric on $M$ is associated to the contact form as in Definition \ref{associateddef}.
%Assume that the Riemannian volume form is a constant multiple of the contact volume form as in Proposition \ref{derivationprop}.
Let $m_0$ be an arbitrary $H^{s-1}$ function on $M$.
Then the velocity field
\begin{equation}\label{Udef}
U := (\eta, \Lambda) \mapsto (\Ssemiop \contactlapinv)_{(\eta, \Lambda)}\Big(e^{-(n+2)\Lambda} m_0\Big)
\end{equation}
defined on the group $\Diffssemicontact(M) = \{ (\eta, \Lambda)\in \Diff^s(M)\ltimes H^s(M)\, \vert \, \eta^*\contact = e^{\Lambda} \contact \} $
is $C^{\infty}$. Hence for any $H^{s+1}$ function $f_0$, there is an $H^s$ geodesic $(\eta(t), \Lambda(t))$ through the identity defined on some (possibly infinite) interval $(-t_b, t_e)$ with $H^s$ initial velocity $(u_0,\lambda_0) = \Ssemiop f_0$.
%(\Scontact f_0, \Reeb(f_0))$.
\end{theorem}

\begin{proof}
The main idea is that first-order twisted differential operators such as $\Ssemiopstar_{(\eta, \Lambda)}\colon T_{(\eta,\Lambda)}\Diffssemi(M)\to H^{s-1}(M)$ are always smooth as a function of $(\eta, \Lambda)$, as described in \cite{ebinmarsden}. We repeat the argument here for the reader's convenience.

We are dealing with an operator $X_{\eta}(h):= X(h\circ\eta^{-1})\circ\eta$, where $h$ is an $H^s$ function and $X$ is a first-order differential operator with smooth coefficients.
For any such operation we have
%Now the vector fields $\Reeb$, $P_k$, and $Q_k$ are all smooth on $M$, and for any smooth vector field $X$ and $H^s$ function $h$ and $H^s$ diffeomorphism $\eta$, we have
\begin{equation}\label{twistedfirstorder}
X_{\eta}(h)(p) = X(h\circ \eta^{-1}) \circ \eta\big|_p  = dh_p\big( (D\eta_p)^{-1} X_{\eta(p)}\big) \quad \forall p\in M.
\end{equation}
If $X$ is a smooth vector field, then the composition $\eta\mapsto X\circ\eta$ is smooth in $\eta$ as long as $\eta\in H^s$ with $s > \tfrac{1}{2}\dim{M} + 1$~\cite{ebinthesis}. In addition the operation $\eta\mapsto (D\eta)^{-1}$ is smooth on the group of diffeomorphisms $\eta$ of the same Sobolev class, since it can be expressed in terms of multiplication (the cofactors) and division by a nowhere-zero function. Since multiplication of $H^{s-1}$ functions is also smooth in each component, the expression $\eta\mapsto X_{\eta}(h)$ given by \eqref{twistedfirstorder} is a smooth function of $\eta\in \Diffs$ and $h\in H^s(M)$.

Recall that elements $T_{(\eta,\Lambda)}\Diffssemi(M)$ are of the form $(v,\rho)$ where $v\in H^s(M,TM)$ and $\rho\in H^s(M,\mathbb{R})$ with $v(p)\in T_{\eta(p)}M$ for each $p\in M$. We can express $v$ in terms of the frame $\{P_k, Q_k, \Reeb\}$ from Proposition \ref{associatedprop} as
\begin{equation}\label{Hstanvec}
v = a \Reeb\circ\eta + \sum_k \big[b_k (P_k\!\circ\!\eta) + c_k (Q_k\!\circ\!\eta)\big]
\end{equation}
for some $H^s$ coefficient functions $\{a, b_k, c_k\}$.
%Recalling that the right translation differential is given by $dR_{(\eta,\Lambda)^{-1}}(v,\rho) = (v\circ\eta^{-1}, \rho\circ\eta^{-1})$, we see
The formal adjoint $\Ssemiopstar$ of $\Ssemiop$ is easy to compute, and thus we find that the twisted operator
$$(\Ssemiopstar)_{(\eta,\Lambda)} = dR_{(\eta,\Lambda)} \circ \Ssemiopstar \circ dR_{(\eta,\Lambda)^{-1}}\colon T_{(\eta,\Lambda)}\Diffssemi(M) \to H^{s-1}(M)$$ looks as follows:% on $(v,\rho)\in T_{(\eta,\Lambda)}\Diffssemi(M)$ with $v$ expressed as in \eqref{Hstanvec}:
%\begin{equation}\label{Ssemiopstarexplicit}
%\begin{split}
\begin{multline}\label{Ssemiopstarexplicit}
(\Ssemiopstar)_{(\eta,\Lambda)}(v,\rho) =
a + \textstyle \sum_k \big((\diver{Q_k}) b_k - (\diver{P_k}) c_k\big)
\\-\Reeb(\rho\circ\eta^{-1})\circ\eta
+ \textstyle\sum_k \big[ Q_k(b_k\circ\eta^{-1})\circ \eta - P_k(c_k\circ\eta^{-1})\circ\eta\big].
\end{multline}
This is thus smooth in the coefficients $(a, b_k, c_k, \rho) \in H^s(M,\mathbb{R})$ by the computation above.

The restriction of each $(\Ssemiopstar)_{(\eta,\Lambda)}$ to $T_{(\eta,\Lambda)}\Diffssemicontact(M)$ is still smooth in $(\eta,\Lambda)$ since $\Diffssemicontact(M)$ is a smooth submanifold of $\Diffssemi(M)$. On this subspace $\Ssemiopstar$ is an isomorphism since $\contactlap=\Ssemiopstar \Ssemiop$ is an isomorphism from $H^{s+1}(M)$ to $H^{s-1}(M)$. The operation  which inverts a linear operator in a vector space is of course smooth, and thus $(\Ssemiopstar)_{(\eta,\Lambda)}^{-1}(h) = (\Ssemiop\contactlapinv)_{(\eta,\Lambda)}(h)$ is smooth in $(\eta,\Lambda)$ for any $h\in H^{s-1}(M)$. The other operations appearing in \eqref{Udef} involve only multiplication and composition with smooth functions, and thus the vector field $U$ is smooth on $\Diffssemicontact(M)$.

Existence of solutions then follows from the usual existence of a local flow for smooth vector fields on smooth Hilbert manifolds, via a Picard iteration argument (see e.g., Lang~\cite{lang}).
\end{proof}

The argument in Theorem \ref{mainthm1} gives existence of short-time solutions $\big(\eta(t),\Lambda(t)\big) \in \Diffssemicontact(M)$ starting at the identity $(\id,0)$ for any initial velocity $(u_0,\lambda_0) = \Ssemiop f_0$ and any $H^{s+1}$ function $f_0$.

Since the curve $(\eta,\Lambda)$ is an integral curve of a smooth vector field, we have smooth dependence on time $t$, and thus the velocity $(\dot{\eta}(t), \dot{\Lambda}(t))$ is an element of $T_{(\eta,\Lambda)}\Diffssemicontact(M)$. Right-translating to the identity, we obtain an $H^s$ vector field $u(t) = \dot{\eta}(t)\circ\eta(t)^{-1}$ which solves the Euler-Arnold equation \eqref{main}.

With this we can proceed to construct a smooth exponential map for $\Diffssemicontact(M)$, as follows.

\begin{corollary}\label{riemannianexp}
Under the conditions of Theorem \ref{mainthm1}, there is a smooth Riemannian exponential map which takes sufficiently small tangent vectors $\Ssemiop f_0\in T_{\id}\Diffssemicontact(M)$ to the time-one solution $\big(\eta(1), \Lambda(1)\big)\in \Diffssemicontact(M)$. By the inverse function theorem on Hilbert manifolds, this exponential map is locally invertible. Hence sufficiently close elements of $\Diffssemicontact(M)$ may be joined by a unique minimizing unit-speed geodesic.
\end{corollary}

\begin{proof}
To obtain this map, we examine the dependence of
$(\eta(t),\Lambda(t))$ on its initial data.  First we note that if
\begin{equation}\label{Utildadef}
\widetilde{U} := (\eta, \Lambda, u_0, \lambda_0) \mapsto (\Ssemiop \contactlapinv)_{(\eta, \Lambda)}\Big(e^{-(n+2)\Lambda} m_0\Big)
\end{equation}
as in \eqref{Udef}, then $\widetilde{U}$ is smooth in all its arguments, so for any $t$, $(\eta(t), \Lambda(t))$ is a smooth function of $(u_0, \lambda_0,t).  $  This function is defined on a neighborhood of $(\tilde{0},0)$ in $T_{id}\Diffssemicontact(M) \times \mathbb{R}$ where $\tilde{0}$ is the zero vector in $T_{id}\Diffssemicontact(M)$.
Thus there is a ball $B_{2\delta}$ about $\tilde{0}$ of radius $2\delta$ and an interval $(-\epsilon, \epsilon)$ such that $(\eta(t), \Lambda(t))$ the solution of
\eqref{Udef}, with initial data in $B_{2\delta}$, is defined for $t \in (-\epsilon, \epsilon).$  But since $(\eta(t),\Lambda(t))$ is a geodesic, we find that for any fixed $r$, $ t \rightarrow (\eta(rt),\Lambda(rt))$ is also a geodesic which of course is defined for
$|t| < \epsilon/r.$ Hence for any initial vector in $B_{\delta r},$ we get a geodesic defined for $|t|<2.$ The value of the exponential map is then defined to be $(\eta(1),\Lambda(1)).$
\end{proof}

Of course there is an isomorphism between $\Diffssemicontact(M)$ and $\Diffscontact(M)$ (obtained by simply forgetting about the scaling $\Lambda$), and we may thus use this result to discuss the Riemannian exponential map directly on the contactomorphism group $\Diffscontact(M)$, if desired. The extension to $\Diffssemicontact(M)$ is only to make the technical details work out more easily.

We have used the conservation law in Proposition \ref{conservationlawprop} to prove local existence of solutions; it also implies that the only thing that can go wrong with global existence is that $\eta$ fails to be a diffeomorphism because the Jacobian determinant $\Jac{(\eta)}$ approaches zero or infinity in finite time. This is the same behavior one sees in a typical one-dimensional nonlinear hyperbolic equation. Intuitively we expect that as long as $\Jac{(\eta)}$ satisfies an estimate of the form $a\le \Jac{(\eta(t))} \le b$ for $0\le t\le T$, then the momentum $m(t)$ will be a globally bounded function. Since $m(t) = f(t) - \Laplacian f(t)$, a $C^0$ bound on $m$ roughly implies a $C^2$ bound on $f$,
% in essence, we conclude a $C^2$ bound on $f$
which leads to estimates on all Sobolev norms of the velocity field $u$ as in \cite{BKM}.
%, in much the same way as Kato~\cite{kato} proves global existence for two-dimensional hydrodynamics and the first author~\cite{ebinsymplectic} proves global existence for symplectohydrodynamics in any dimension.
Since $\Jac{(\eta)} = \exp{[(n+1)\Lambda]}$ for any $(\eta,\Lambda)\in \Diffsemicontact(M)$ by Lemma \ref{jacobianlemma}, we can write the global existence condition in terms of the function $\Lambda$, and this gives a ``Beale-Kato-Majda''-style criterion for global existence as in \cite{BKM}. As before we will work with an associated Riemannian metric just to simplify the notation, though the result does not depend on this assumption.

\begin{theorem}\label{mainthm2}
Let $M$ be a compact contact manifold of dimension $2n+1$ as in Theorem \ref{mainthm1} with associated Riemannian metric. Let $(u(t),\lambda(t)) = \Ssemiop f(t)$ be a solution of the Euler-Arnold equation \eqref{main} (defined a priori only for short time) with $f(t)\in H^{s+1}(M)$ for some $s>n+3/2$. Then the solution exists up to time $T$ if we have
\begin{equation}\label{BKMcontacto}
\int_0^T \lVert \Reeb(f)(t)\rVert_{L^{\infty}} \, dt = C<\infty \qquad \text{for all $t\in [0,T]$.}
\end{equation}
\end{theorem}

\begin{proof}
%One direction is easy. If $\int_0^T \lVert \Lambda(t)\rVert_{L^{\infty}} \, dt = \infty$, then we must have a sequence of times $t_n\ne T$ such that $\lVert \Lambda(t_n)\rVert_{L^{\infty}} \to \infty$. For each $n$ choose $p_n\in M$ such that $\lvert \Lambda(t_n, p_n)\rvert \ge \tfrac{1}{2} \lVert \Lambda(t_n)\rVert_{L^{\infty}}$. Choosing a subsequence and reindexing, we may assume that $p_n\to p\in M$.
%Let $M_0 = \int_0^T \lVert \Lambda(t)\rVert_{L^{\infty}} \, dt$.
Since the solution $(u(t),\lambda(t))$ exists as long as $(\eta(t),\Lambda(t))$ does, and since $(\eta(t),\Lambda(t))$ solves a smooth ordinary differential equation on $\Diffssemicontact(M)$, it is sufficient to show that $\lVert u(t)\rVert_{H^s}$ and $\lVert \lambda(t)\rVert_{H^s}$ remain bounded; hence it is sufficient to show that $\lVert f(t)\rVert_{H^{s+1}}$ is bounded on $[0,T]$, where $f(t)$ is the stream function. We will use the conservation law \eqref{conservation} to achieve this.

First we note that since $\Jac{(\eta)} = e^{(n+1)\Lambda}$ by Lemma \ref{jacobianlemma} and $\Lambda$ satisfies $\frac{\partial}{\partial t} \Lambda(t,x) =
\Reeb(f)(t,\eta(t,x))$, our assumption \eqref{BKMcontacto} implies that
%$$   e^{-(n+1)\int_0^T \Reeb(f)(t,\eta(t,x)) \,dt} \le \Jac(\eta(t,x)) \le e^{(n+1) \int_0^T \Reeb(f)(t,\eta(t,x)) \, dt $$
%$$ e^{-(n+1) \int_0^T \lVert \Reeb(f)(t)\rVert_{L^{\infty}}} \le \Jac(\eta(t,x)) \le e^{(n+1) \int_0^T \lVert \Reeb(f)(t)\rVert_{L^{\infty}} \, dt}$$
%\begin{equation}\label{jacobianbound}
$$ e^{-(n+1) C} \le \Jac{\eta(t)} \le e^{(n+1)C}.$$
%\end{equation}
Using \eqref{conservation}, we obtain that
$\lvert m(t,\eta(t,p))\rvert \le \lvert m_0(p)\rvert e^{(n+2)C}$ for all $p\in M$, and in particular $\lVert m(t)\rVert_{C^0} \le e^{(n+2)C} \lVert m_0\rVert_{C^0}$.
Using this we will find a bound for $df(t)$, or equivalently for $\Ssemiop f(t)$.

Recall from Proposition \ref{associatedprop} that the contact Laplacian $\contactlap$ is an isomorphism from $H^{s+1}(M)$ to $H^{s-1}(M)$. From the theory of elliptic operators (\cite{taylor}, Chapter 7, Proposition 2.2 for $M = \mathbb{R}^n$ and Chapter 7, Section 12 for $M$ any compact manifold) we find that its inverse is a pseudodifferential operator whose Schwartz kernel $k: M \times M \rightarrow \mathbb{R} $ is smooth off the diagonal and obeys the estimate
\begin{equation} \label{pseudo-est} |D_p^{\beta} k(p,q)| \leq K \dist(p,q)^{-2n+1-|\beta|} \quad {\rm for} \quad 2n+|\beta| > 1, \end{equation}
where $D_p^{\beta}$ means a $\beta$-order derivative operator with respect to the first variables of $k$ and $\dist(p,q) $ is the distance from $p$ to $q$ defined by the Riemannian metric on $M$.\footnote{In the inequalities that follow $K$ will always denote some positive constant, but it may be different in different inequalities.}

From this and the fact that $m = \contactlap f = f-\Laplacian f$, we can estimate
$$
f(p) = \int_M k(p,q) m(q) \, d\mu(q)
$$
and
\begin{equation}
\label{df} df(p) = \int_M d_p k(p,q) m(q) d\mu(q)
\end{equation}
where $d_p$ is the differential with respect to the $p$-variables and $d\mu(q)$ indicates integration with respect to $q$ using the Riemannian volume element of $M.$\footnote{In the sequel we will sometimes write $dk$ for $d_p k.$}
We have $|k(p,q)| \leq K \dist(p,q) ^{-2n+1}$ and $|d_p k(p,q)| \leq K \dist(p,q) ^{-2n},$ so both of these integrals are  bounded by a constant times
$\| m(t)\|_{C^0}$ %= \| \contactlap f(t)\|_{C^0}$
which, as we have seen, is bounded in time. With this and formula \eqref{Scontactdarboux} we see that $u = \Sop f$ is bounded uniformly in time as well.
%\textcolor{red}{There is no need for a separate argument in case the manifold has dimension $2$!}

We proceed to seek a time-uniform Lipschitz bound for $df$, but in fact we will be able to find only a quasi-Lipschitz bound, as we shall now explain.  Since $M$ is compact, we can find a positive $\smally$ such that each point of $M$ has a normal coordinate neighbourhood ball of radius at least $\smally.$

Fix $p\in M$. Then for any $q \in M$ with $\dist(p,q) < \smally$, we have a unique minimal geodesic $\mingeodesic$ parameterized so that $\mingeodesic(0) = p$ and $\mingeodesic (1) = q$.
Fix such a $q$. We let  $ b = |\mingeodesic'(\tau)|,$ so that $b = \dist(p,q)$. Also we parallel translate $ df(p)$ along $\mingeodesic$ to get some $df'(q) \in T^* _q M.$ Then we shall estimate
$|df(q) - df'(q)|$ where $|\; \: |$ is the norm on $T^* _q M.$
To do this we use the formula \eqref{df} for $df$, and we parallel translate each $d_p k(p,r)$ for $r\in M$ along $\mingeodesic$ to get $dk'(q,r) \in T^* _q M.$  In this way we get $$ df'(q) = \int_M dk'(q,r) m(r) \,d\mu(r)$$ and we find
$$df'(q) - df(q) = \int_M (dk'(q,r) - dk(q,r)) m(r) \,d\mu(r)$$

Following \cite{kato}, Lemma 1.4, we split up this integral as follows:
Let $\Sigma = B_{2b} (p),$ the ball of radius $2b$ about $p$.
Then $\int_M = \int_{\Sigma} + \int_{M\backslash \Sigma}$. For the integral over $\Sigma$, we have the estimate
$$\int_{\Sigma} dk'(q,r) m(r) \,d\mu(r) \leq K \lVert m \rVert_{C^0} \int_{\Sigma} \dist(p,r)^{-2n} \,d\mu(r)  \leq 2bK \lVert m\rVert_{C^0}.$$
Also $\Sigma \subset B_{3b}(q)$ and
$$ \int_{\Sigma} dk(q,r) m(r)\,d\mu(r) \leq
K \lVert m \rVert_{C^0} \int_{\Sigma} \dist(q,r)^{-2n} \,d\mu(r)
\leq 3bK \lVert m \rVert_{C^0}. $$
Combining these we get
\begin{equation} \label{sigest} \int_{\Sigma} \big(dk'(q,r) - dk(q,r)\big) m(r) \,d\mu(r) \leq 5bK \lVert m \rVert_{C^0}.
\end{equation}

The estimate of the integral over $M\backslash\Sigma$
%$\int_{N-\Sigma} (dk'(y,z) - dk(y,z)) \Laplacian_{\theta} f(z) dz  $
is more subtle.  If we let $P(\tau)$ denote parallel translation along $\mingeodesic$ from $\mingeodesic(\tau)$ to $\mingeodesic(1)$,
%(\textcolor{red}{this was $\mingeodesic(b)$ before but I think you wanted it to be $1$ since the geodesic is not unit-speed})
so that
$P(\tau): T^*_{\mingeodesic(\tau)} M \rightarrow T^*_q M,$ we find that
%\begin{eqnarray}
\begin{align*}
dk'(q,r) - dk(q,r) &= P(0)\big(d k(p,r)\big) -P(1)\big(dk(q,r)\big) \\
&= -\int_0^1 P(\tau) \nabla_{\mingeodesic'(\tau)}dk(\mingeodesic(\tau),r)\,d\mu(r).
\end{align*}
%\end{eqnarray}
Since parallel translation is an isometry, we therefore have
%\begin{eqnarray}
\begin{align*}
 |dk'(q,r) -dk(q,r)| &\le \max_{0\leq \tau \leq 1} \{|\nabla_{\mingeodesic'(\tau)} dk(\mingeodesic(\tau),r)|\}\\
 &\le K \max_{0\le \tau\le 1} \lvert \mingeodesic'(\tau)\rvert \dist\big(\mingeodesic(\tau), r\big)^{-2n-1}.
% &\le Kb \dist(q,r)^{-2n-1}
% &\le Kb(\dist(p,r)-b)^{-2n-1},
\end{align*}
Now $\lvert \mingeodesic'(\tau)\rvert = b$ for all $\tau$, and by the triangle inequality we have for all $\tau\in [0,1]$ that
$$ \dist\big(\mingeodesic(\tau),r\big) \ge \dist(p, r) - \dist\big(p, \mingeodesic(\tau)\big) \ge \dist(p,r) - b.$$
We conclude that
$$ \lvert dk'(q,r) - dk(q,r)\rvert \le Kb \lvert \dist(p,r) - b\rvert^{-2n-1}$$
whenever $r\in M\backslash \Sigma$, for any $q$ such that $\dist(p,q)=b<\smally$.

Let $R= {\rm diam}(M)<\infty$ (since $M$ is compact).  Then
%\begin{eqnarray}
\begin{align*}
\int_{M\backslash\Sigma} |dk'(q,r)-dk(q,r)| \, d\mu(r) &\le  Kb\int_{M\backslash\Sigma} (\dist(p,r)-b)^{-2n-1}d\mu(r)\\
&\le Kb \int_{2b}^R \frac{\rho^{2n}}{(\rho-b)^{2n+1}} \, d\rho
\end{align*}
(possibly modifying $K$).
For $2b\le \rho\le R$ we know that $1\le \frac{\rho}{\rho-b}\le \frac{3}{2}$, so we can overestimate
$$ \int_{M\backslash\Sigma} |dk'(q,r)-dk(q,r)| \, d\mu(r) \le K'b \int_{2b}^R (\rho - b)^{-1} \, d\rho \le K'b \log{(R/b)}.$$
%&\le K'b \int_{2b}^R (\rho-b)^{-1}\,d\rho \\
%&= K'b(\log(R-b) - \log b) \\
%&\le K'b \log(R/b)
%\end{align*}

Combining this with \eqref{sigest} we find
$$
\int_M |dk'(q,r) -dk(q,r)| \,d\mu(r) \leq 5K\dist(p,q) + K'\dist(p,q) \log(R/\dist(p,q))
$$
for $\dist(p,q) < \smally.$  By increasing $K$ and $K'$ we simplify this inequality to
$$
\int_M |dk'(y,z)-dk(y,z) | \,d\mu(r) \leq K \dist(p,q)(1+ \log(R/\dist(p,q)).
$$
Thus we find that if $\dist(p,q)<\smally$ and $df'(q)$ is the parallel transport of $df(p)$ along the minimizing geodesic from $p$ to $q$, then
\begin{equation} \label{quasi}
|df'(q) - df(q)| \leq K\dist(p,q)(1+\log(R/\dist(p,q))
\end{equation}
where $K$ is independent of $t$ as before. With this inequality we say that $df$ is \emph{quasi-Lipschitz}.
Also since $df(p)$ is bounded independently of $p$ and $t$ we find that by further increasing $K$
we get \eqref{quasi} for all $p,q \in M;$ that is, we can drop the restriction $\dist(p,q) < \smally.$
As a consequence note that for any $\gamma<1$ we have
\begin{equation}
\lvert df'(q)-df(q)\rvert \le K \delta(p,q)^{\gamma}
\end{equation}
since the logarithm grows slower than any power. Hence $df$ is uniformly $C^{\gamma}$ for any $\gamma<1$.

%STOPPED HERE ON PLANE

Since $u = \Sop f$, our bound for $df$ gives the same quasi-Lipschitz bound for $u$, again uniformly in $t$.
Using this bound we can find a positive $\alpha$ for which the flow $\eta(t)$ of $u(t)$ is $C^{\alpha}.$
However we also need to show that $\eta(t)^{-1}$ is $C^{\alpha}$.  Fortunately we can do this by the
same method: given any fixed $t_0 \in [0,T)$, we define a time dependent vector field $v$ on
$M$ by $v(t) = -u(t_0-t)$. Then if $\sigma$ is the flow of $v$, it is easy to see that the maps
$t\mapsto \sigma\big(t, \eta(t_0,x)\big)$ and $t\mapsto\eta(t_0-t,x)$ satisfy the same differential equation,
and since $\sigma\big(0, \eta(t_0,x)\big) = \eta(t_0,x)$, they must be equal for all times $t\in [0,t_0]$.
Hence in particular $\sigma(t_0)=\eta(t_0)^{-1}.$ We proceed to show that $\sigma(t_0)$ is $C^{\alpha}$ for some $\alpha>0$;
the fact that each $\eta(t)$ is $C^{\alpha}$ is similar.

Fix $p$ and $q$ in $M$. Let $\mingeodesic\colon[0,t_0]\times [0,1]$ be the map such that for each $t\in [0,t_0]$, the curve
$\tau\mapsto \mingeodesic(t,\tau)$ is the minimal geodesic between $\sigma(t,p)$ and $\sigma(t,q)$
with $\mingeodesic(t,0)=\sigma(t,p)$ and $\mingeodesic(t,1)=\sigma(t,q)$.
Define
$$
\phi(t) = \dist\big(\sigma(t,p), \sigma(t,q)\big) = \int_0^1 \Big\lvert \frac{\partial \mingeodesic}{\partial \tau}(t,\tau)\Big\rvert \,d\tau.
$$
Then
\begin{align*}
\phi'(t) &= \frac{d}{dt}\left(\int_0^1 \Big\lvert \frac{\partial \mingeodesic}{\partial \tau}(t,\tau)\Big\rvert \, d\tau\right) \\
&= \int_0^1 \frac{1}{\lvert \frac{\partial \mingeodesic}{\partial \tau}(t,\tau)\rvert} \, \Big\langle
\frac{\partial\mingeodesic}{\partial \tau}(t,\tau), \frac{D}{\partial t} \frac{\partial \mingeodesic}{\partial \tau}(t,\tau)\Big\rangle \, d\tau
\end{align*}
But
$\tfrac{D}{\partial t} \tfrac{\partial\mingeodesic}{\partial \tau} = \tfrac{D}{\partial \tau} \tfrac{\partial\mingeodesic}{\partial t}$ by general properties of surface maps (e.g., \cite{lang} Chapter XIII, Lemma 5.3), and $\tfrac{D}{\partial \tau} \tfrac{\partial\mingeodesic}{\partial \tau} = 0$ since each $\tau\mapsto \mingeodesic(t,\tau)$ is a geodesic, and thus an integration by parts yields
$$
\phi'(t) = \frac{1}{\lvert \tfrac{\partial \mingeodesic}{\partial \tau}(t,\tau)\rvert} \Big\langle \frac{\partial \mingeodesic}{\partial \tau}(t,\tau), \frac{\partial \mingeodesic}{\partial t}(t,\tau)\Big\rangle\Big\vert_{\tau=0}^{\tau=1},
$$
using the fact that $\lvert \tfrac{\partial\mingeodesic}{\partial\tau}(t,\tau)\rvert$ is constant in $\tau$ since $\mingeodesic$ is a geodesic in $\tau$.

Now $\frac{\partial \mingeodesic}{\partial \tau}$ is parallel along $\mingeodesic$, and since the parallel transport
$P$ from $\mingeodesic(0)$ to $\mingeodesic(1)$ preserves inner products, we have
$$ \langle \tfrac{\partial \mingeodesic}{\partial \tau}(t,0), v\big(t, \sigma(t,p)\big)\rangle =
\langle \tfrac{\partial\mingeodesic}{\partial \tau}(t,1), Pv\big(t,\sigma(t,p)\big)\rangle. $$
Thus
\begin{align*}
\phi'(t) &= \frac{1}{\lvert \tfrac{\partial \mingeodesic}{\partial \tau}(t,\tau)\rvert} \big\langle \tfrac{\partial \mingeodesic}{\partial \tau}(t,1), v\big(t, \sigma(t,q)\big) - Pv\big(t, \sigma(t,p)\big)\big\rangle \\
&\le \lvert v\big(t, \sigma(t,q)\big)
-Pv\big(t,\sigma(t,p)\big)\rvert.
\end{align*}
Now since $v$ like $u$ is quasi-Lipschitz on $[0,t_0]$, we have a constant $K$ such that
\begin{equation}
\label{phidiffineq} \phi'(t) \leq K \phi(t)\left(1+ \log \left(\frac{R}{\phi(t)}\right)\right),
\end{equation}
where the constants $R$ and $K$ do not depend on $t$ or $t_0$.

We proceed to estimate $\phi(t)$.  Let $\psi(t) = \log(\phi(t)/R)$ so $\psi' = \phi'/\phi.$
Then from \eqref{phidiffineq} we get $\psi' \leq K(1-\psi)$.
Integrating this we find
$$\psi(t) \leq \psi(0)e^{-Kt} + 1-e^{-Kt}.$$
Exponentiating this inequality and noting that $\phi(0) \leq R$ we find that
$$\frac{\phi(t)}{R}\leq \left(\frac{\phi(0)}{R}\right)^{e^{-Kt}}e^{1-e^{-Kt}}
\leq \left(\frac{\phi(0)}{R}\right)^{e^{-KT}}e
$$
Thus $\phi(t) \leq R^{1-e^{-KT}} e \phi(0)^{e^{-KT}}$,
so letting $\alpha = e^{-KT}$ and $L = eR^{1-e^{-KT}}$,
we get $\phi(t) \leq L \phi(0)^{\alpha}$
and hence
\begin{equation} \label{rhoineq}
\dist(\sigma(t_0,p),\sigma(t_0,q)) \leq L \dist(p,q)^{\alpha}.
\end{equation}
The constants $L$ and $\alpha$ do not depend on the choice of $t_0$, so the estimate \eqref{rhoineq}
holds for all $t_0\in [0,T)$.
From \eqref{rhoineq} and the conservation law from Proposition \ref{conservationlawprop} in the form
$$ m(t,p) = m_0\big(\sigma(t,p)\big) \exp{\big[ -(n+2)\Lambda\big(t,\sigma(t,p)\big)\big)]},$$
%$$ \lvert m(t,p)\rvert \le e^{(n+2)C} \lvert m_0(p)\rvert,$$
%$$ \contactlap f(t) = \contactlap f_0 \circ \sigma(t)$$
we conclude that $m$ is H\"older continuous as follows: since $df$ is $C^{\gamma}$ for
any $\gamma$, so is $\Reeb(f)$, and thus $\Lambda(t,p) = \int_0^t \Reeb(f)\big(\tau, \eta(\tau,p)\big) \,d\tau$
is H\"older continuous as a composition of H\"older continuous functions. Now $\Lambda\circ \sigma$ and
$m_0\circ \sigma$ are also
H\"older continuous since $\sigma$, $\Lambda$, and $m_0$ are. Finally the product of H\"older continuous
functions is still H\"older continuous (for a possibly smaller exponent), so we find that $m(t)$
is bounded uniformly in $t$ in
$C^{\alpha}(M,\mathbb{R})$ for some $\alpha>0$, and thus by standard elliptic theory we get
$f(t)$ bounded in $C^{2+\alpha}(M,\mathbb{R})$,
from which it follows that $\Sop f$ is bounded in $C^{1+\alpha}(TM)$.

We now need a $C^1$ bound on $m(t) = \contactlap f(t)$, which we obtain as follows:
computing the gradient of both sides of \eqref{main},
we have $$\nabla m_t + \nabla_u\nabla m + [\nabla, \nabla_u] m + (n+2) m \nabla \Reeb(f) + (n+2) \Reeb(f) \nabla m = 0,$$
which implies that
\begin{align*}
\frac{d}{dt} \lvert \nabla m(t,\eta(t,x))\rvert
&\le \lvert [\nabla, \nabla_u] m\rvert (t,\eta(t,x)) \\
&\qquad\qquad + (n+2) \lvert m(t,\eta(t,x))\rvert \lvert \nabla \Reeb(f)(t,\eta(t,x))\\
&\qquad\qquad + (n+2) \lvert \Reeb(f)(t,\eta(t,x))\rvert \lvert (\nabla m)(t,\eta(t,x))\rvert \\
&\le K \lVert u(t)\rVert_{C^1} \lVert m(t)\rVert_{C^1} \\
&\qquad\qquad+ (n+2) \lVert m(t)\rVert_{C^0} \lVert f(t)\rVert_{C^2} \\
&\qquad\qquad + (n+2) \lVert f(t)\rVert_{C^1} \lVert m(t)\rVert_{C^1}.
\end{align*}
Gronwall's inequality then implies that
$$ \lVert m(t)\rVert_{C^1} \le \lVert m_0\rVert_{C^1}
\exp{\left( K \int_0^t \lVert f(\tau)\rVert_{C^2} \, d\tau\right)}$$
for some constant $K$, and since $f(t)$ is bounded in $C^{2+\alpha}$, we know it is also bounded in $C^2$; thus $m(t)$ is bounded in $C^1$.

Now we show that $f(t)$ is bounded in the $H^{s+1}$ topology, or equivalently that $m(t) = \contactlap f(t)$ is bounded in $H^{s-1}.$
Since $\partial_t m = -u(m) - (n+2) m\Reeb(f)$,
%$$ \partial_t \contactlap f = -S_{\contact}f (\Laplacian_{\contact} f)
%$$
for $s=1$ we have
\begin{align*}
\frac{d}{dt} \int_M m^2 \mu &= -\int_M u(m^2) \, d\mu - 2(n+2)\int_M m^2 \Reeb(f) \, d\mu \\
&= \int_M m^2 \big( \diver{u} - (2n+4) \Reeb(f)\big) \, d\mu \\
&= -(n+3) \int_M m^2 \Reeb(f) \, d\mu,
\end{align*}
using \eqref{divergenceSop},
so that
$$ \int_M m(t)^2 \, d\mu \le e^{(n+3)C} \int_M m_0^2 \, d\mu.$$
If $s>1$, then taking $s-1$ spatial derivatives\footnote{ ~Powers of $\nabla$ are defined in a standard way: for any function $f,$ $\nabla f$ is a section of $TM$, $\nabla^2 f$ is a section of $TM \otimes T^*M$ and for any $k$, $\nabla^k f$ is a section of $TM \otimes (T^*M)^{k-1}$.  Inner products are defined by the induced Riemannian metric.} we get
\begin{equation}\label{intyparts}
\begin{split}
\frac{d}{dt} \int_M |\nabla^{s-1} m |^2 \mu & =
-2 \int_M \langle \nabla_{u} \nabla^{s-1} m, \nabla^{s-1} m \rangle \, d\mu \\
&\qquad -2(n+2) \int_M \Reeb(f) \langle \nabla^{s-1}m, \nabla^{s-1}m\rangle \, d\mu
\\
 &\qquad -2 \int_M \langle[\nabla^{s-1}, \nabla_u]m,\nabla^{s-1} m\rangle \, d\mu \\
 &\qquad - 2(n+2) \int_M \langle \nabla^{s-1}m, [\nabla^{s-1}, \Reeb(f)]\nabla^{s-1}m\rangle \, d\mu,
\end{split}
\end{equation}
where $[ \;,\;]$ denotes the commutator.

As before, the first two terms in \eqref{intyparts} reduce to
$$ -(n+3) \int_M \Reeb(f) \lvert \nabla^{s-1}m\rvert^2 \, d\mu,$$
and for the last two terms we use the standard estimate
\begin{equation}\label{calculusinequality}
\| \nabla^k (hg) - h \nabla^k g\|_{H^0} \leq K \big(\|h\|_{H^k} \|g\|_{C^0} + \|\nabla h \|_{C^0} \|g\|_{H^{k-1}}\big),
\end{equation}
(with $k=s-1$) which can be found in \cite{taylor} Chapter 13, Proposition 3.7.
%Here we use $k=s-1$.
For the first commutator we choose $h = u$ and $ g = \nabla m$
and obtain
\begin{multline*}
\int_M \big\langle[\nabla^{s-1}, \nabla_u]m,
\nabla^{s-1} m\big\rangle \mu \\
\leq K \Big(
\lVert u\rVert_{H^{s-1}} \lVert m\rVert_{C^1} + \|u\|_{C^1} \| m \|_{H^{s-1}}\Big)\lVert m\rVert_{H^{s-1}}.
\end{multline*}
We already have bounds for $\lVert m(t)\rVert_{C^1}$ and for $\lVert u(t)\rVert_{C^1}$, and since $\lVert u(t)\rVert_{H^{s-1}} \lesssim \lVert m(t)\rVert_{H^{s-2}}$ we have a bound for the first commutator in terms of $\lVert m(t)\rVert^2_{H^{s-1}}$.

To bound the second commutator in \eqref{intyparts}, we use \eqref{calculusinequality} again with $h=\Reeb(f)$ and $g=m$ to obtain
\begin{multline*}
\int_M \langle \nabla^{s-1}m, [\nabla^{s-1}, \Reeb(f)]\nabla^{s-1}m\rangle \, d\mu \\
\le \tilde{K} \Big( \lVert f\rVert_{H^s} \lVert m\rVert_{C^0} + \lVert f\rVert_{C^2} \lVert m\rVert_{H^{s-2}}\Big)\lVert m\rVert_{H^{s-1}},
\end{multline*}
and we already have bounds for each of these. Combining and overestimating, we obtain
$$ \frac{d}{dt} \lVert \nabla^{s-1}m\rVert^2_{L^2} \le (n+3) \lVert \Reeb(f)\rVert_{C^0} \lVert \nabla^{s-1}m\rVert_{L^2}^2 + \overline{K} \lVert m\rVert_{C^1} \lVert m\rVert^2_{H^{s-1}},$$
which leads to a bound on $\lVert m(t)\rVert_{H^{s-1}}$ on $[0,T]$ by Gronwall's inequality.
This completes the proof.
\end{proof}

In the next section we will analyze a special case where one can obtain global existence essentially for free.

\section{Special cases and other aspects}

\subsection{Quantomorphisms}\label{quantosection}

In some situations we care more about the contact form than the contact structure. In this case the appropriate group to consider is the group of \emph{quantomorphisms} given by
$$ \Diffexcontact(M) = \{ \eta\in \Diff(M) \, \vert \, \eta^*\contact = \contact\}.$$
We may identify this group with the subgroup
$$ \{ (\eta, \Lambda)\in \Diffsemicontact(M) \, \vert \, \Lambda = 0\}.$$
Every quantomorphism preserves the volume form by Lemma \ref{jacobianlemma}. A quantomorphism also preserves the Reeb field:
infinitesimally if $\Reeb(f)=0$ and $u=\Sop f$,
then
$$[\Reeb, u] = [\Sop 1, \Sop f] = \Sop \{1,f\} = \Sop(\Reeb(f)) = 0$$
by \eqref{contactpoisson};
the noninfinitesimal proof works as in \cite{ratiuschmid}.
As a result we have
\begin{equation}\label{quantotangent}
T_{\id}\Diffexcontact(M) = \{ \Sop f\, \vert \, \Reeb(f)\equiv 0\}.
\end{equation}
The ``padded quantomorphism group,'' viewed as a subgroup of the padded contactomorphism group, consists of
$$ \widetilde{\Diffexcontact}(M) = \{ \Ssemiop f\, \vert \, \Reeb(f) \equiv 0 \} = \{ (\Sop f, 0\}\, \vert \, \Reeb(f) \equiv 0\}.$$
The following example shows that the quantomorphism group structure depends greatly on the properties of the Reeb field.

\begin{exmp}\label{torus3d}
On $M=\mathbb{T}^3 = (\mathbb{R}/2\pi \mathbb{Z})^3$ with coordinates $(x,y,z)$ and with contact form $\contact = \sin{z} \, dx + \cos{z} \, dy$, the Reeb field is $\Reeb=\sin{z} \, \partial_x + \cos{z} \, \partial_y$. Every quantomorphism must preserve the Reeb field, but the Reeb field has nonclosed orbits whenever $\tan{z}$ is irrational, and hence any function which is constant on the orbits must actually be a function only of $z$. It is then easy to see that the identity component of $\Diffexcontact(\mathbb{T}^3)$ consists of diffeomorphisms of the form
$$ \eta(x,y,z) = \big(x+p(z)\sin{z} + p'(z) \cos{z}, y+p(z)\cos{z}-p'(z)\sin{z}, z\big)$$
for some function $p\colon S^1\to \mathbb{R}$. This group is abelian, so any right-invariant metric will actually be bi-invariant, and all geodesics will be one-parameter subgroups.
\end{exmp}

The only way to get an interesting quantomorphism group is if the Reeb field happens to have all of its orbits closed and of the same length. In this case the contact manifold must be related to a symplectic manifold by a \emph{Boothby-Wang fibration}~\cite{boothbywang}. We say that the contact form is \emph{regular}, following Ratiu and Schmid~\cite{ratiuschmid}. If this happens, then there is a symplectic manifold $N$ given as the quotient space of $M$ by the orbits, with a map $\pi\colon M\to N$ and a symplectic form $\omega$ on $N$ such that $\pi^*\omega = d\contact$. The best-known example is the Hopf fibration of $S^3$ over $S^2$. When the contact form is regular, the tangent space to $\Diffexcontact(M)$ may be identified with the space of functions $f\colon M\to \mathbb{R}$ such that $\Reeb(f)=0$.

Omori~\cite{omori} proved (Theorem 8.4.2) that if $\contact$ is regular, then $\Diffsexcontact(M)$ is a smooth Hilbert submanifold of $\Diffssemicontact(M)$. Hence the Riemannian metric \eqref{contactometric} induces a Riemannian metric on $\Diffsexcontact(M)$, and the geodesic equation on the submanifold is obtained by the tangential projection of the full geodesic equation \eqref{main} on $\Diffssemicontact(M)$.
We now prove that this submanifold is totally geodesic by showing that the second fundamental form vanishes.

\begin{proposition}\label{totallygeodesic}
Suppose $M$ is a contact manifold with an associated Riemannian metric as in Definition \ref{associateddef}, and suppose that the Reeb field $\Reeb$ is a Killing field.
If $\contact$ is a regular contact form on $M$, then $\widetilde{\Diffexcontact}(M)$ is a totally geodesic submanifold of $\Diffsemicontact(M)$. Hence any solution of the Euler-Arnold equation \eqref{main} such that $\Reeb(f_0)=0$ will have $\Reeb(f(t))=0$ for all time.
\end{proposition}

\begin{proof}
It is an elementary result in Riemannian geometry that a submanifold is totally geodesic (i.e., geodesics which start in the submanifold remain there) if and only if the second fundamental form vanishes identically. To show that it vanishes, it is sufficient to show that
$\langle \nabla_{\tilde{u}}\tilde{u}, \tilde{v}\rangle = 0$ whenever $\tilde{u}$ is tangent to the submanifold and $\tilde{v}$ is orthogonal to it. For a right-invariant Riemannian metric on a Lie group, we have
$\nabla_{\tilde{u}}\tilde{u} = \ad_{\tilde{u}}^*{\tilde{u}}$, so it is sufficient to show that $\langle \tilde{u}, \ad_{\tilde{u}}\tilde{v}\rangle = 0$ whenever $\tilde{u}\in T_{\id}\widetilde{\Diffexcontact}(M)$ and $\tilde{v}\in T_{\id}\Diffsemicontact(M)$ is orthogonal to $T_{\id}\widetilde{\Diffexcontact}(M)$.

To be precise we write $v = \Ssemiop g$ and $u=\Ssemiop f$ where $\Reeb(f)=0$. We want $v$ orthogonal to $\Ssemiop h$ whenever $\Reeb(h)=0$, which gives a condition on $g$ as follows: we want
$$ \llangle \Ssemiop g, \Ssemiop h\rrangle = \int_M h \contactlap g \, d\mu = 0$$
whenever $\Reeb(h)=0$. Since $\Reeb$ is assumed to be Killing, it commutes with $\Laplacian$ and hence with $\contactlap$; hence we also have
$\int_M gq\, d\mu =0$ whenever $\Reeb(q)=0$.

 %and since $\Reeb$ is divergence-free we must have $\contactlap g = \Reeb(j)$ for some function $j$.
%Thus $g = \contactlapinv\Reeb(j)$. We note that since the metric is associated, we have $\contactlap = 1 - \Laplacian$, and since $\Reeb$ is a Killing field, it commutes with both $\Laplacian$ and $\contactlap$.

Now if $\Reeb(f)=0$,
%and $g = \contactlapinv\Reeb(j)$ for some function $j$,
then we compute that $\tilde{u}=\Ssemiop f$ and $\tilde{v}=\Ssemiop g$ satisfy
$$
\llangle \tilde{u}, \ad_{\tilde{u}}\tilde{v}\rrangle = \int_M \langle \Ssemiop f, \Ssemiop \{f, g\} \rangle \, d\mu
= \int_M \{f,g\} m \, d\mu,$$
where $m=\contactlap f$.
Since $\Reeb(f)=0$ we have from \eqref{contactpoisson} that $\{f,g\} = u(g)$ where $u=\Sop f$, so that
$$ \llangle \tilde{u}, \ad_{\tilde{u}}\tilde{v}\rrangle = \int_M m u(g) \, d\mu = -\int_M g \diver{(m u)} \, d\mu
= -\int_M g u(m) \, d\mu,$$
since $\diver{u}=(n+1)\Reeb(f)=0$ by Proposition \ref{associatedprop}.

Now $g$ is orthogonal to any function which is Reeb-invariant, so we will have $\llangle \tilde{u}, \ad_{\tilde{u}}\tilde{v}\rrangle = 0$
for $\tilde{v} = \Ssemiop g$ as long as we know $\Reeb\big(u(m)\big)=0$ whenever $m=\contactlap f$ and $u=\Sop f$ for an $f$ with $\Reeb(f)=0$.
Since $\Reeb = \Sop 1$, we have
\begin{align*}
\Reeb\big(u(m)\big) &= \Sop 1\big(\Sop f(m)\big) \\
&= [\Sop 1, \Sop f](m) + \Sop f\big( \Sop 1(m)\big) \\
&= \Sop\{1,f\}(m) + \Sop f\big( \Reeb(m)\big).
\end{align*}
Now by formula \eqref{contactpoisson} we have $\{1,f\} = \Sop 1(f) - f \Reeb(1) = \Reeb(f) = 0$,
and since $\Reeb$ commutes with $\contactlap$ we have $\Reeb(m) = \Reeb(\contactlap(f)) = \contactlap(\Reeb(f)) = 0$.
We conclude that $\llangle \tilde{u}, \ad_{\tilde{u}}\tilde{v}\rrangle = 0$, and thus the second
fundamental form of $\widetilde{\Diffexcontact}(M)$ is zero.
\end{proof}

Proposition \ref{totallygeodesic} has the easy corollary that if the contact form is regular, any solution of the Euler-Arnold equation \eqref{main} for which $\Reeb(f_0)=0$ will automatically have global solutions in time, using Theorem \ref{mainthm2}.

\begin{corollary}\label{globalquanto}
Suppose $M$ is a compact contact manifold with associated Riemannian metric satisfying Definition \ref{associateddef}, and such that the contact form $\contact$ is regular with the Reeb field $\Reeb$ a Killing field of the metric. Then any solution of \eqref{main} such that $\Reeb(f_0)\equiv 0$ will have $\Reeb(f(t))=0$ whenever it is defined, and hence by Theorem \ref{mainthm2} the solution will exist for all time.
\end{corollary}

If $M$ is three-dimensional (so that the Boothby-Wang quotient is two-dimensional, and its volume form is the symplectic form), the Euler-Arnold equation on the quantomorphism group takes the form
\begin{equation}\label{fplane}
m_t + \{f,m\} = 0, \qquad m = f-\Laplacian f,
\end{equation}
where $\{\cdot, \cdot\}$ is the standard Poisson bracket. We may rescale the metric on $M$ so that the Reeb field has a different constant length $\alpha$, and in this case the momentum takes the form $m = \alpha^2 f - \Laplacian f$. Thus the Euler-Arnold equation on the quantomorphism group of $M$ is the quasigeostrophic equation in $f$-plane approximation on $N$, as in Holm-Zeitlin~\cite{holmzeitlin} and Zeitlin-Pasmanter~\cite{zeitlinpasmanter}; here $\alpha^2$ is the Froude number.

An alternative approach to the quantomorphism group is to view it as a central extension of the group $\Diffham(N)$ of Hamiltonian diffeomorphisms of the symplectic manifold $N$; this approach is used in Ratiu-Schmid~\cite{ratiuschmid} and is also taken in the references \cite{HT, GV, GT}.
%More recently several authors have observed that the quantomorphism group is the proper configuration
%space for the geodesic Vlasov equations (our geodesic equation is essentially a special case of this).
%Holm-Tronci~\cite{HT} proposed a general quadratic Hamiltonian
%on the quantomorphism group and studied its moments. Gay-Balmaz and Vizman~\cite{GV} computed the momentum maps for this equation, and Gay-Balmaz and Tronci~\cite{GT} found an interesting totally geodesic subgroup.
Smolentsev~\cite{smolentsevquanto} computed the curvature tensor of the quantomorphism group under the same assumptions.

A more sophisticated version of the quasigeostrophic equation is the $\beta$-plane approximation, for which the evolution equation for $f(t,x,y)$ takes the form
\begin{equation}\label{betaplanequasi}
(-\alpha^2 f + \Laplacian f)_t + \{f,\Laplacian f\} + \beta \partial_xf = 0,
\end{equation}
where $\alpha$ and $\beta$ are constants.
Vizman~\cite{vizmanquasigeostrophic} derived this equation as the Euler-Arnold equation of a central extension of the group of Hamiltonian diffeomorphisms, in the case that $\alpha=0$. However the same central extension applied to the group of quantomorphisms yields \eqref{betaplanequasi}. We can obtain global existence for these equations in exactly the same way as in Theorem \ref{mainthm2},
since there is a potential vorticity which is transported and there is no stretching.

Explicitly, given a Reeb-invariant function $\psi$ on $M$, and two Reeb-invariant
functions $p$ and $q$ on $M$ such that the contact bracket satisfies\footnote{All such functions descend to
the Boothby-Wang quotient $N$, so that all we want is that $\{\overline{p},\overline{q}\} = \overline{\psi}$ for the standard
Poisson bracket on $N$, where $\overline{p}$, $\overline{q}$, and $\overline{\psi}$ are the quotient functions on $N$.} $\{p,q\} = \psi$,
let us define a cocycle on $\Diffexcontact(M)$ by the formula
\begin{equation}\label{groupcocycle}
B(\eta, \xi) = \int_M p (q\circ \eta + q\circ\xi - q\circ\eta\circ\xi - q) \, d\mu.
\end{equation}
The corresponding cocycle on the Lie algebra is
\begin{equation}\label{algebracocycle}
b(u,v) = -\int_M p [u,v](q) \, d\mu,
\end{equation}
and $p$ and $q$ are related by $\{p,q\} = \psi$.

%EXPAND
\begin{proposition}
If $M$ is a contact manifold with an associated metric and a regular contact form, and $N$ is its Boothby-Wang quotient
with the Riemannian metric prescribed so that the projection $\pi\colon M\to N$ is a Riemannian submersion, then
on the Lie algebra $T_{\id}G$ consisting of $T_{\id}\Diffexcontact(M)$ with central extension defined by \eqref{algebracocycle},
the Euler-Arnold equation $\dot{u} + \ad_u^*u=0$ reduces to
\begin{equation}\label{quasigeo}
\vorticity_t + \{\overline{f},\vorticity\} = 0,\qquad \vorticity = \Laplacian \overline{f}-\alpha^2 \overline{f} - \beta \psi,
\end{equation}
in terms of $\overline{f}\colon \mathbb{R}\times N\to \mathbb{R}$ and the Poisson bracket $\{\cdot, \cdot\}$ on $N$.
On $N=\mathbb{R}^2$ with $\psi(x,y) = y$, we obtain the standard $\beta$-plane approximation \eqref{betaplanequasi}.
\end{proposition}

\begin{proof}
Write $u\in T_{\id}G$ as $u = (\Sop f, c)$ where $f = \overline{f}\circ\pi$ for some $\overline{f}\colon N\to \mathbb{R}$
and $c\in \mathbb{R}$. It is sufficient to compute the inner product of the Euler-Arnold equation with an arbitrary $v = (\Sop g, d)$,
which takes the form $\llangle \dot{u}, v\rrangle + \llangle u, \ad_uv\rrangle = 0$. Here we have
$$
\ad_uv = \big(\ad_{\Sop f}\Sop g, b(\Sop f, \Sop g)\big) = \big(-\Sop \{f,g\}, -\int_M p \Sop\{f,g\}(q)\big),$$
in terms of the contact bracket \eqref{contactpoisson} on $M$. However we note that $\{f,g\} = \{\overline{f}, \overline{g}\}\circ\pi$
for $\Reeb$-invariant functions in terms of the quotient Poisson bracket on $N$. In addition since $p$, $q$, $f$, and $g$ are all
$\Reeb$-invariant, we have
\begin{multline*}
\int_M p \Sop\{f,g\}(q) \, d\mu = \int_M p\big\{\{f,g\}, q\big\} \, d\mu = L\int_N \overline{p} \big\{\{\overline{f},\overline{g}\},\overline{q}\big\} \, d\nu \\
= -L\int_N \{\overline{p},\overline{q}\} \{\overline{f},\overline{g}\} \, d\nu = -L\int_N \overline{\psi} \{\overline{f},\overline{g}\} \, d\nu
= L\int_N \overline{g} \{ \overline{f}, \overline{\psi}\} \, d\nu,
\end{multline*}
where $\nu$ is the volume form on $N$ and $L$ is the length of the Reeb field orbit.

From here we easily compute
\begin{align*}
0&= \llangle u_t, v\rrangle + \llangle u, \ad_uv\rrangle \\
&= \int_M g \contactlap f_t \, d\mu + c_t d - \int_M \contactlap f \{f,g\} \, d\mu
+ L\int_N \overline{g} \{ \overline{f}, \overline{\psi}\} \, d\nu \\
&= L \int_N \overline{g}\big( \alpha^2 \overline{f}_t- \Laplacian \overline{f} + \{\overline{f}, \alpha^2 f-\Laplacian f\} + \{\overline{f}, \overline{\psi}\}\big) \, d\nu + c_t d.
\end{align*}
This is zero for every $g$ and $d$ if and only if $c_t=0$ and equation \eqref{quasigeo} holds.
\end{proof}

In general, conservation of the potential vorticity $\vorticity$ in equation \eqref{quasigeo}
implies global existence just as in Corollary \ref{globalquanto}.

\subsection{Other aspects of the contactomorphism equation}

In this section we will remark on some interesting features of the contactomorphism equation: in particular some special infinite-energy one-parameter solutions, peakon solutions supported on submanifolds of codimension one (analogous to the standard peakons in the Camassa-Holm equation), and a few conservation laws which are analogous to the first few conserved quantities in the infinite hierarchy in the Camassa-Holm equation.

\subsubsection{Infinite-energy solutions}

Although we have proved Theorems \ref{mainthm1} and \ref{mainthm2} under the assumption that the contact manifold $M$ is compact, the equation \eqref{main} makes sense even if $M$ is not compact, as long as the stream function $f$ has compact support or decays sufficiently quickly. The situation we discuss here on $\mathbb{R}^3$ does not satisfy these properties, but gives a one-dimensional equation that can be studied in some detail, and helps illustrate the similarities between equation \eqref{main} and the Camassa-Holm equation.

We work with the standard Darboux contact form $\contact = dz - y\,dx$ on $\mathbb{R}^3$, with Reeb field $\Reeb = \frac{\partial}{\partial z}$. A natural Riemannian metric in this case is given by
$$ ds^2 = (dz-y\,dx)^2 + dx^2 + dy^2,$$
since in this case the metric is associated, as discussed in Example \ref{contactexamples}.
Consider a stream function $f$ of the form $f(t,x,y,z) = zg(t,y)$ for some function $g$; then $m(t,x,y,z) = z [g(t,y)-g_{yy}(t,y)]$, and $u=-zg_y\, \partial_x + y g_y\,\partial_y + z (g-yg_y) \, \partial_z$. We can check that the Lie subalgebra consisting of such vector fields generates a totally geodesic submanifold, or simply verify that stream functions of this form give solutions of \eqref{main}; the equation that $g(t,y)$ must satisfy ends up being
\begin{equation}\label{onedimensionalversion}
g_t - g_{tyy} + 4g^2 - 4gg_{yy} = ygg_{yyy} - yg_y g_{yy}.
\end{equation}
The problem is that no such stream function can have finite $H^1$ norm on $\mathbb{R}^3$, and hence results that one can prove about \eqref{onedimensionalversion} do not necessarily apply to \eqref{main}, in much the same way that infinite-energy solutions of the equations of two-dimensional hydrodynamics may blow up in finite time~\cite{childressetal} although finite-energy solutions cannot. However the simpler one-dimensional case can give clues to the behavior of the higher-dimensional situation.

Sarria and the second author proved the following theorem~\cite{prestonsarria}. It gives a clue as to the role of the sign of the momentum in blowup, although the results are not directly applicable to our case.

\begin{theorem}
Define $\phi_0(y) = g_0(y)-g_0''(y)$, and assume $g_0$ is $C^2$ and satisfies the decay condition $\phi_0(y)=O(1/y^2)$ as $\lvert y\rvert \to \infty$. Then there is a $T>0$ such that there is a unique solution of \eqref{onedimensionalversion} with $g(0,y) = g_0(y)$, and $y\mapsto g(t,y)$ is $C^2$ for each $t$. If $\phi_0$ is nonnegative, then so is $g_0$, and solutions exist globally in time. On the other hand, if $g_0$ is even and negative, then solutions blow up at some time $T$ and $g(t,y)\to -\infty$ as $t\nearrow T$ for every $y\in \mathbb{R}$.
\end{theorem}

Essentially what happens here is that the integral of the momentum $\phi = g-g_{yy}$ is not conserved (as it would be in the finite-energy case).\footnote{Note that the actual momentum is $m(t,x,y,z)=z\phi(t,y)$, so the integral of it over the unbounded $z$ domain is infinite.} Instead one can show that $\frac{d}{dt} \int_{\mathbb{R}} \phi(t,y) \, dy \le 0$.  The momentum gets transported by the flow $\gamma$ of the vector field $y\mapsto yg(t,y)$, in the sense that $\phi\big(t,\gamma(t,y)\big) = \phi_0(y)y^5 \gamma_y(t,y)/\gamma(t,y)^5$, the analogue of the momentum conservation law \eqref{conservation}. Thus if $\phi_0(y)$ never changes sign, then neither does $\phi(t,y)$ for any $t>0$. Furthermore we can prove that if $\phi$ never changes sign, we have the global bound
\begin{equation}\label{C1C0}
\lvert g_y(t,y)\rvert \le \lvert g(t,y)\rvert \quad \text{for all $t$ and $y$}.
\end{equation}
If $\phi_0$ is nonnegative, then the $L^1$ norm of $\phi$ decays in time, and we obtain a global $L^1$ bound on $\phi$ which is sufficient to obtain an $L^{\infty}$ bound on $g$ and hence also $g_y$, which gives global existence. On the other hand if $\phi$ and hence $g$ are nonpositive and symmetric about the origin initially, they will remain so for all time, and we can compute the bound $g_t(t,0) \le -\sqrt{6} g(t,0)^2$. This implies blowup of $g(t,0)$ in finite time along with blowup of $g(t,y)$ for all other $y$ as a consequence of the differential inequality \eqref{C1C0}.

\subsubsection{Peakon solutions}

Euler-Arnold equations of the form \eqref{main} satisfying conservation laws that can be expressed in the form \eqref{conservation} have weak solutions where the momentum $m$ is supported on some lower-dimensional collection of subsets. For example the Camassa-Holm equation \eqref{camassaholm} has solutions which take the form $m(t,\onedparam) = \sum_{k=1}^n p_k(t) \delta\big(\onedparam-q_k(t)\big)$, where the functions $p_k(t)$ and $q_k(t)$ satisfy a Hamiltonian system. The velocity field $u = (1-\partial_{\onedparam}^2)^{-1}m$ is continuous and has cusps at each of the points $q_k(t)$. Similarly one may consider singular solutions of the ideal Euler equation; in two dimensions the solutions with vorticity concentrated on points is a well-known model, and in three dimensions one may consider either the equation for vortex filaments (that is, vorticity concentrated on curves which evolve in space) or for vortex sheets (where vorticity is concentrated on surfaces). See Khesin~\cite{khesinvortexsingular} for a general discussion of such solutions in ideal fluids, and Holm et al.~\cite{holmpeakons} for a discussion in the case of the Camassa-Holm and EPDiff equations. Note that there are subtleties here: for example while peakons in the Camassa-Holm equation are genuine weak solutions due to the rather mild singularity in the one-dimensional Green function, point and filament vorticity models are less mathematically rigorous due to the unboundedness of the higher-dimensional Green functions; in such cases one may need to apply a renormalization procedure to obtain a closed system. %See Marchioro-Pulvirenti~\cite{marchioropulvirenti} for an analysis of such issues.

The general approach is to integrate the conservation law to the form \eqref{conservation}, and assume that $m$ is a sum of delta-function distributions supported on submanifolds of various codimensions. Depending on the codimension, we may end up with a velocity field $u$ that is well-defined even though $m$ is singular, and this velocity field generates a flow $\eta$ along which the supports of $m$ will move. For example in the Camassa-Holm equation the Schwarz kernel of the operator $(1-\partial_{\onedparam}^2)^{-1}$ is given by $K(\onedparam_1,\onedparam_2) = \tfrac{1}{2} e^{-\lvert \onedparam_1 - \onedparam_2\rvert}$ which is bounded, and hence if $m$ is a sum of delta functions, then the velocity field will always be a weak solution in $H^1(\mathbb{R})$. More generally one can consider distributions supported on sets of higher codimension, although only when the codimension is one can we expect the Schwartz kernel to be bounded. General singular solutions of this form for the symplectomorphism Euler-Arnold equation~\cite{ebinsymplectic} were studied by Khesin~\cite{khesinsymplectic}, particularly in the case of point symplectic vortices.

%The most interesting submanifolds of a contact manifold $M$ of dimension $2n+1$ are the Legendrian submanifolds of dimension $n$, defined by the condition that their tangent spaces are everywhere in the kernel of the contact structure. For example if $M$ is of dimension $3$, then we have Legendrian knots which are curves $s\mapsto \gamma(s)$ such that $\contact(\gamma(s))=0$ for all $s$. The major application of such submanifolds is in the Cartan theory of PDEs, where a first-order partial differential equation in a single function $u$ in $n$ variables $(x^1,\ldots, x^n)$ is represented by the solution of $F(x^1,\ldots, x^n, u_{x_1}, \ldots, u_{x^n}, u)=0$. We may write this as the algebraic equation $F(x^1,\ldots, x^n, p^1,\ldots, p^n, z) = 0$, and subsets in $\{x^1,\ldots, x^n, p^1,\ldots, p^n, z\}$-space satisfying this equation represent solutions of the PDE if and only if the contact form $\contact = dz - \sum_k p^k \, dx^k$ vanishes along them.

Heuristically, the conservation law \eqref{conservation} %m\big(t, \eta(t,p)\big) = m_0(p)/\Jac(\eta(t,p))^{(n+2)/(n+1)},
 implies that if an initial momentum $m_0$ is concentrated on a submanifold $\Gamma_0$, then the momentum will be concentrated on a curve $\Gamma(t)$ of submanifolds of the same dimension for all time $t$. The Lagrangian flows $\eta(t)$ will be contactomorphisms for all $t$. Thus for example we could consider the evolution of a Legendrian submanifold (of dimension $n$) or, in case $n=1$ so that $M$ is three-dimensional, we could consider the evolution of an overtwisted disc~\cite{EKM}.

Here we will just compute a simple example in the situation of Examples \ref{contactexamples} and \ref{torus3d}, on the torus $\mathbb{T}^3 = (\mathbb{R}/2\pi\mathbb{Z})^3$.  Recall the contact form is $\contact = \sin{z} \, dx + \cos{z} \, dy$, the Reeb field is $\Reeb=\sin{z} \, \partial_x + \cos{z} \, \partial_y$, and the associated metric is the flat Euclidean metric $ds^2 = dx^2 + dy^2 + dz^2$. As in Example \ref{contactexamples} we have
\begin{equation}\label{Soptorus3d}
\Sop f = (f\sin{z} + f_z\cos{z}) \, \tfrac{\partial}{\partial x} + (f\cos{z} - f_z \sin{z}) \, \tfrac{\partial}{\partial y} + (f_y \sin{z} - f_x \cos{z}) \, \tfrac{\partial}{\partial z}.
\end{equation}
Consider one-parameter stream functions $f$ such that $m=f-\Laplacian f$ is zero except on a two-dimensional submanifold. For example if $f=f(z)$, then we solve the equation $f(z)-f''(z)=0$ for $-\pi < z<\pi$, demanding only continuity at $z=\pi=-\pi$ (but not differentiability). We obtain $f(z) = \cosh{z}$, for which the velocity field given by \eqref{Soptorus3d} is
$$ u(x,y,z) = (\sin{z} \cosh{z} + \sinh{z} \cos{z}) \, \tfrac{\partial}{\partial x} + (\cosh{z} \cos{z} - \sinh{z}\sin{z}) \, \tfrac{\partial}{\partial y}.$$
As $z\searrow -\pi$ we obtain $u(x,y,-\pi) = \sinh{\pi} \, \tfrac{\partial}{\partial x} - \cosh{\pi} \, \tfrac{\partial}{\partial y}$, while as $z\nearrow \pi$ we have $u(x,y,\pi) = -\sinh{\pi} \, \tfrac{\partial}{\partial x} - \cosh{\pi} \, \tfrac{\partial}{\partial y}$. Thus the velocity field has a jump discontinuity across the $2$-torus $z=\pi=-\pi$, where the flow shears horizontally. Since the surface is not moved by the velocity field (i.e, the velocity field is tangent to the singular surface), this is a steady solution of equation \eqref{main}.

This jump discontinuity of tangential components across the singular surface is typical. The basic model is a function satisfying $f-f_{zz}=\delta(z)$ on $\mathbb{R}^3$ with a singularity at $z=0$, given by $f(x,y,z) = \tfrac{1}{2} e^{-\lvert z\rvert}$, which is differentiable in the $x$ and $y$ directions and continuous but not differentiable in the $z$ direction. Suppose $f$ is a solution of $f - \Laplacian f = 0$ on the complement of some surface $\Gamma_0$, and let $N$ denote a unit normal vector to $\Gamma_0$ at a point $p\in \Gamma_0$. Let $\{\Reeb, P, Q\}$ be an orthonormal frame as in Proposition \ref{associatedprop}, so that $u = \Sop f = f \Reeb + \Reeb \times \grad f$ (using the cross product where $\Reeb \times P = Q$, $P\times Q= \Reeb$, and $Q\times \Reeb = P$). Then
$$ \langle u, N\rangle = f\langle \Reeb, N\rangle + \langle N, \Reeb\times \grad f\rangle = f \langle \Reeb, N\rangle + \langle N\times \Reeb, \grad f\rangle.$$
Since this differentiates $f$ only in the direction $N\times \Reeb$ which is tangent to the surface, it is continuous on $\Gamma_0$.
Thus if we want to interpret the peakon solutions as an evolution equation for a singular surface, we should consider only the evolution of the normal vector field to the surface, which is well-defined.

The only way $u= \Sop f$ is well-defined on the entire manifold is if $u$ never differentiates in the normal direction, which means $\Reeb$ would have to be parallel to $N$ everywhere along the singular surface. In other words the tangent plane to the surface would have to be an integral surface of the contact structure, which is of course impossible by the definition of the contact structure. Hence the shear in the velocity field is a characteristic feature of contact geometry.

%
%\begin{proposition}
%Suppose $M$ is a three-dimensional contact manifold with an associated Riemannian metric. Suppose $f_0$ is a function satisfying
%$m_0=f_0-\Laplacian f_0=0$ on $M\backslash \Gamma_0$ where $\Gamma_0$ is a smooth submanifold. If $u_0 = \Sop f_0$, then $u_0$ is continuous
%on $M\backslash \Gamma_0$ and has a jump discontinuity on $\Gamma_0$ in a direction tangential to $\Gamma_0$.
%\end{proposition}
%
%\begin{proof}
%To study singular points of $u_0$, it is sufficient to look at a small neighborhood of a point $p_0\in \Gamma_0$. In such a small neighborhood
%the singular part of the Schwartz kernel $K$ of the operator $(1-\Laplacian)^{-1}$ is proportional to $K_0(p,q) \simeq \frac{1}{\dist(p,q)}$.
%Let $B_r$ be a small geodesic ball around $p_0$; then we have
%$$ f(p_0) = \int_M K(p_0,q) m(q) \, d\mu(q) = \int_{\Gamma_0} K(p_0,q) \, dA = \int_{B_r\cap \Gamma_0} K_0(p_0,q) \, dA + \text{nonsingular terms.}$$
%However the singular term looks in coordinates such that $B_r \cap \Gamma_0$ is a flat disc like
%$$ \int_{B_r\cap \Gamma_0} K_0(p,q) \, dA = \int_0^{r} \int_0^{2\pi} \frac{1}{r} r\, dr \,d\theta = 2\pi r.$$
%This term is continuous

\subsubsection{Conservation laws}

The Camassa-Holm equation~\eqref{camassaholm} is a completely integrable system, which implies that there are sufficiently many conservation laws that one can use them to form action-angle variables in which the flow is linear. The most obvious conservation law is the $H^1$ energy given (see e.g., \cite{lenellscamassaholm}) by $C_1[m] = \int_{S^1} mu \, d\onedparam = \int_{S^1} u^2 + u_{\onedparam}^2 \,d\onedparam$. Another conservation law (which can be used to generate another compatible Hamiltonian structure) is $C_2[m] = \int_{S^1} u^3 + uu_{\onedparam}^2 \, d\onedparam$. Others include the quantities $C_{-1}[m] = \int_{S^1} \sqrt{m} \, d\onedparam$ and $C_0[m] = \int_{S^1} m\,d\onedparam$, where the integral of $\sqrt{m}$ is taken over only the subset where $m$ is positive (a similar law works for $\sqrt{-m}$ on the subset where $m$ is negative). Some of these laws, in particular $C_1$, $C_0$, and $C_{-1}$, generalize in a very obvious way to our higher-dimensional case \eqref{main}, and we will present those laws here. It is not clear whether the other conservation laws work here or what form they should take, but it would certainly be interesting to obtain a form of complete integrability for equation \eqref{main}; as it is there are very few examples of completely integrable  systems in any dimension higher than one; see for example \cite{KLMP2}.

\begin{proposition}\label{conservationlaws}
Suppose $f$ is a solution of the equation \eqref{main} on a contact manifold of dimension $2n+1$, with an associated Riemannian metric for which $m=f-\Laplacian f$ and for which $\Reeb$ is a Killing field. Then the following quantities are constant along any solution:
$$ C_0 = \int_M m \, d\mu, \quad C_{-1,\pm} = \int_M m_{\pm}^r \, d\mu, \quad \text{ and } C_1 = \int_M mf \, d\mu,$$
where $r=\frac{n+1}{n+2}$ and $m_+$ and $m_-$ denote the positive and negative parts of $m$ (i.e., $m_+(x) = m(x)$ if $m(x)>0$ and $m_+(x)=0$ otherwise).
\end{proposition}

\begin{proof}
To prove these, we observe that if $u=S_{\contact}f$, then $\diver{u} = (n+1) \Reeb(f)$ by Proposition \ref{associatedprop}. We can thus write \eqref{main} as
\begin{equation}\label{divergenceform}
m_t + u(m) + \frac{n+2}{n+1} (\diver{u}) m = 0.
\end{equation}
Integrating both sides over $M$ we obtain
\begin{align*}
\frac{d}{dt} \int_M m \, d\mu &= -\frac{1}{n+1} \int_M (\diver{u}) m \, d\mu \\
&= -\int_M f\Reeb(f) \, d\mu + \int_M \Reeb(f) \Laplacian(f) \, d\mu \\
&= \tfrac{1}{2} \int_M f^2 \diver{\Reeb} \, d\mu - \int_M \langle \grad f, \grad \Reeb(f)\rangle \, d\mu \\
&= -\tfrac{1}{2} \int_M \lvert \grad f\rvert^2 \diver{\Reeb} \, d\mu \\
&= 0,
\end{align*}
since $\diver{\Reeb}=0$ by Proposition \ref{associatedprop}, and $\Reeb$ commutes with the gradient since it preserves the metric (so that $\grad \Reeb(f) = \nabla_{\Reeb} \grad f$).

For the other conservation law, note that on a domain for which $m$ is positive we can write equation \eqref{divergenceform} as
$$ \partial_t(m^r) + \diver{( m^r u)} = 0,$$
which immediately leads to the conservation law
$$ \int_{\Omega(t)} m(t)^r \, d\mu = \text{constant},$$
where $\Omega(t)$ is a domain transported by the flow. But of course by the conservation law \eqref{conservation}, the region $\Omega(t)$ on which $m(t)$ is positive is transported by the flow. The same argument leads to conservation of $m_-$.

Finally, the fact that $C_1$ is conserved follows easily from the fact that $C_1 = \lVert f(t)\rVert^2_{H^1} = \lVert \Ssemiop f(t)\rVert^2$, which is precisely the
energy of the velocity vector of a geodesic (which is always conserved). Alternatively we could derive it directly via integration by parts, as for the other laws.
\end{proof}

\makeatletter \renewcommand{\@biblabel}[1]{\hfill#1.}\makeatother

\end{document}